\documentclass{amsart}
\usepackage{amssymb,graphicx,rotating,multirow,multicol}
\usepackage{amsmath,amsthm,amsfonts,amscd}
\usepackage[toc,page]{appendix}
\usepackage{caption,comment,marginnote}
\usepackage{enumitem}
\usepackage[mathscr]{euscript}
\usepackage{epstopdf}
\usepackage{float}
\usepackage{graphicx}
\usepackage{hyperref}

\usepackage{epigraph}
\swapnumbers
\newtheorem{theorem}{Theorem}[subsection]

\newtheorem{lemma}[theorem]{Lemma}

\newtheorem{cor}[theorem]{Corollary}
\newtheorem{proposition}[theorem]{Proposition}
\theoremstyle{remark}
\newtheorem{remark}[theorem]{Remark}
\numberwithin{equation}{subsection}
\newcommand{\K}{\ensuremath{\mathbb{K}}}

{\catcode`\@=11 \gdef\mnote#1{\marginpar{\footnotesize
 \tolerance\@M\spaceskip2.6\p@ plus10\p@ minus.9\p@\rm#1}}}
\def\Dg:{\endgraf{\bf Dg:\enspace}\ignorespaces}

\let\Bbb\mathbb
\let\Cal\mathcal

\def\sm{\smallsetminus}

\newcommand{\be}{\begin{equation}}
\newcommand{\ee}{\end{equation}}

\def\Z{\Bbb Z}
\def\R{\Bbb R}
\def\C{\Bbb C}
\def\Q{\Bbb Q}
\def\N{\Bbb N}
\def\P{\Bbb P}

\def\B{\mathcal B}
\def\CC{\mathcal C}

\def\X{\mathcal X}
\def\W{\Cal N}

\def\SB{\Bbb S}

\def\Im{\operatorname{Im}}
\def\Rp#1{\Bbb{RP}^{#1}}

\def\Pin{\operatorname{Pin}}
\def\Spin{\operatorname{Spin}}
\def\rk{\operatorname{rk}}

\def\conj{\operatorname{conj}}

\def\bhv{\operatorname{\varUpsilon}}
\def\q{\operatorname{\widehat q}}
\def\Pic{\operatorname{Pic}}
\def\eff{\operatorname{Eff}}
\def\lay{\operatorname{\Cal L}}
\def\card{\operatorname{card}}

\def\D4{\operatorname{\frak D}}
\def\discr{\operatorname{discr}}

\def\e{\varepsilon}

\def\ee{\boldsymbol{\ell}}
 \def\v{\frak v}

 \def\w{\frak W}

\def\dsum{\bot\!\!\!\bot}
\let\+\dsum
\let\a=\alpha
\let\ge\geqslant % >=
\let\le\leqslant % <=

\let\til\widetilde

\def\FS{\Theta^{(1)}}
\def\SS{\Theta^{(2)}}

\makeatletter
\newcommand{\addresseshere}{%
  \enddoc@text\let\enddoc@text\relax
}
\makeatother

\title[Wall-crossing Invariance]{On wall-crossing invariance of certain sums of Welschinger numbers}
\author[]
{S.~Finashin, V.~Kharlamov}
\makeatletter
\@namedef{subjclassname@2020}{%
  \textup{2020} Mathematics Subject Classification}
\makeatother

\begin{document}
\begin{abstract} We continue our quest for real enumerative invariants not sensitive to changing the real structure and extend
the construction we uncovered previously for
counting
curves of anti-canonical degree $\le 2$ on del Pezzo surfaces with $K^2=1$ to
curves of any anti-canonical degree and on any del Pezzo surfaces of degree $K^2\le 3$.
\end{abstract}

\maketitle

\setlength\epigraphwidth{.47\textwidth}
\epigraph{The obvious answer is always overlooked}{Known as Whitehead's Law \\}

\section{Introduction}\label{intro}
\subsection{Problem formulation} Up to 2000s it was a rather common opinion that there can not exist any reasonable integer valued enumerative geometry over the reals,
and that it is a prerogative of geometry over the field of complex numbers and other algebraically closed fields. The situation is radically changed in 2003, when J.-Y.~Welschinger invented an integer signed count of point-constrained real rational curves on real rational surfaces that
respects the
necessary invariance property to be
preserved under equivariant deformations and equivariant isomorphisms.

Those Welschinger invariants, like most of their subsequent analogs,
turned out to be
sensitive to changes
that
a real structure
experiences under
{\it wall-crossing}.
However, quite soon
after
Welschinger's discovery, some examples appeared
where under an appropriate counting scheme
the invariance under wall-crossing
holds.
Initially, it
was a signed count of real lines on real projective hypersurfaces (see  \cite{S}, \cite{W}, \cite{O-T}, \cite{F-K}),
which includes as the starting case the count of real lines on real cubic surfaces in real projective 3-space.

It is the latter signed count on cubic surfaces that
was recently extended by us
to counting
curves of anticanonical degree 1 and 2
on del Pezzo surfaces of degree $K^2=1$ (see \cite{TwoKinds}, \cite{Combined}).
The invariance under wall-crossing was achieved
there by
combining the original Welschinger invariants with a certain intrinsic $\Pin^-$-structure we attributed to real loci of these surfaces.

This
led us to
a general question:
{\it What kind of real rational surfaces $X$ carry similar $\Pin^-$-structures on $X_\R$ and
what are $($anticanonical $)$ degrees for which the real rational curves in $X$ can be counted
invariantly under wall-crossing~?}

In this paper we
answer this question for del Pezzo surfaces of degree $K^2\le 3$.

\subsection{Main results}
Let $X$ be a real del Pezzo surface of degree $K^2\le 3$ and $\conj:X\to X$ the complex conjugation.
Denote by $\eff(X)\subset \Pic(X)=H_2(X)$ the semigroup of effective divisor classes,
and put $\eff_\R(X)= \eff(X)\cap  \ker(1+\conj_*)$.
If $X_\R\ne \emptyset$ (which is always the case if $K^2$ is odd),
the latter semigroup coincides with the semigroup of divisor classes that can be realized by
a real effective divisor.

For del Pezzo surfaces, $\eff(X)$ and $\eff_\R(X)$ are both
finitely generated. For $K^2=2$ and $3$, they are generated by
{\it lines}, which are, by definition,
embedded $(-1)$-curves of genus 0 (defined over the reals if we speak of {\it real lines}).
For $K^2=1$, these semigroups
have one additional generator, $-K$.

We split both $\eff(X)$ and $\eff_\R(X)$, into
subsets called {\it layers}:
\begin{equation*}
\lay^m(X)= \{ \alpha\in\eff(X)\,|\, -\alpha K=m\}, \quad \lay^m_\R(X)= \{ \alpha \in\eff_\R(X)\,|\, -\alpha  K=m\}.
\end{equation*}
All the layers are finite sets.
They are empty if $m$ is not a positive integer.

For each pair of integers
$m, k$ with $0\le k\le m-1, k= m-1 \mod 2,$ and a collection $\mathbf x\subset X$ of $m-1$ points including
 $k$ real ones and $\frac12(m-k-1)$
pairs of
complex conjugate imaginary points,
we consider
 the set
 $\Cal C_\R(m,k,\mathbf x)$ of real rational irreducible reduced curves $A$, $[A]\in \lay^m_\R(X)$, $\mathbf x\subset A$
that pass through $\mathbf x$.
The sets $\Cal C_\R(m,k,\mathbf x)$ are all finite for a generic choice of $\mathbf x$.

The main object of this paper is the following double sequence of integers,
\begin{equation}\label{main-definition}
N_{m,k}= \sum_{A\in \Cal C_\R(m,k,\bold x)}
i^{\hat q([A])-m^2}w(A),
\end{equation}
where $\hat q : \eff_\R(X)\to \Z/4$ is
a certain quadratic function that  $X$ inherits from its natural embeddings in an appropriate 3-fold (see Section \ref{preliminaries})
and $w(A)$ is the modified {\it Welschinger number} $w(A)=(-1)^{c_A}$ in which $c_A$ denotes the number of
cross-point real nodes in the real locus of $A$.

As is known, for any  $\alpha \in \lay^m_\R(X)$ a partial sum
$W_{\alpha,k}=\hskip-5mm\sum\limits_{A\in \Cal C_\R(m,k,\bold x), [A]=\alpha} \hskip-3mm w(A)$,
which we call the modified {\it Welschinger invariant},
does not depend on a generic choice of $\bold x$ (see \cite{Br}) and is invariant under real deformations of $X$ (see \cite{Welsch}).
The same invariance of $N_{m,k}$ now follows from the real deformation invariance of $\q$ (see Section \ref{preliminaries}).

Some precaution must be however taken in the case $K^2=1$, $m=1$.
Namely, to have such
invariance for numbers $W_{-K,0}$
and thus, for $N_{1,0}$,
we need, in addition to genericness
of $\bold x$, assume also genericness of $X$
(see \cite{revisited}).

Recall that over $\C$ all del Pezzo surfaces with a fixed $K^2$ are deformation equivalent to each other, while over $\R$ two real del Pezzo surfaces having the same $K^2$ are real deformation equivalent if, and only if, their real structures $\conj : X\to X$ are diffeomorphic.

The goal of this paper is to prove that the sequence $N_{m,k}$
has the following strong invariance property.

\begin{theorem}\label{main}
The double sequence  $N_{m,k}$ is the same for all real del Pezzo surfaces $X$ with $X_\R\ne\varnothing$ having a given degree $K^2\le 3$.
\end{theorem}

Since for any $k\ge 2$ and any nonsingular real rational surface $X$ with disconnected real locus $X_\R$ the Welschinger invariants
$W_{\alpha,k}$ vanish
(see \cite{Br}), and since for any $K^2\in \{1,2,3\}$ there exist real del Pezzo surfaces with disconnected $X_\R$, Theorem \ref{main}
implies the following vanishing for $N_{m,k}$.

\begin{cor}\label{vanishing} For all real del Pezzo surfaces
as in Theorem \ref{main}, each of the numbers $N_{m,k}$ with $k\ge 2$ is equal to 0.
\qed
\end{cor}

The numbers $N_{m,k}$ happen to have remarkable properties.
In particular,
for $k = 1$ they turn out to be related in a ``magic way'' with Gromov-Witten invariants.

\begin{theorem}\label{magic-theorem}
For any $m\in \Z_{\ge 1}$ and any real surface $X$ as in Theorem \ref{main},
\begin{equation}\label{mag-formula}
N_{2m,1}= 2^{m-3}\sum_{\alpha\in \lay^m (X)} (e\alpha)^2 GW_{\alpha}
\end{equation}
where $e$ is an arbitrary class $e\in K^\perp$ with $e^2=-2$ and $GW_\alpha$ states for the Gromov-Witten invariant which counts rational irreducible reduced curves in class $\alpha\in\lay^m(X)$ passing through a generic collection of $m-1$ points.
\end{theorem}

%Apart of
Besides
%allowing
%\mnote{F: Besides, without "to perform" to simplify}
% to perform
%numerical calculations of
giving a way to calculate
%\mnote{Kh: not very nice "numerical calculation of numbers" is replaced by "giving a way to calculate"}
the numbers $N_{2m,1}$, this result plays a crucial role
%\mnote{F: "crucial" since "one of the key" looks funny}
%is playing one of the key roles
in deriving
%\mnote{Kh: "two" closed}
%\mnote{F: without "a set"}
%a set of
%two
simple
recursive formulas governing the both sequences, $N_{2m,1}$ and $N_{2m+1,0}$.
\begin{theorem}\label{two-recursive-formulas}For any $n\in \Z_{\ge 1}$ and any real surface $X$ as in Theorem \ref{main},
\begin{equation*}
\begin{aligned}
mK^2N_{m,0}= &2\sum_{j=1}^{n}\binom{n-1}{n-j}j(m-2j)^2N_{m-2j,0}N_{2j,1} \quad\text{for \,\,$ m=2n+1$},\\
mK^2N_{m,1}= &2\sum_{j=1}^{n}\binom{n-1}{n-j}j(m-2j)^2 N_{m-2j,1}N_{2j,1} \quad\text{for \,\,$ m=2n+2$}.
\end{aligned}
\end{equation*}
\end{theorem}

These recursion relations allow to reconstruct the numbers
$N_{m,k}$ from the initial values (see Section \ref{applications})
and to observe their non-vanishing and positivity
(except the case of $N_{2n+1,0}=0$ for $K^2=2$).

\begin{remark}
In the case $K^2=2$, the numbers $N_{m,k}$ vanish for odd $m$ (see Proposition \ref{vanishing-in-degree-2}), and so
in this case the first formula in Theorem \ref{two-recursive-formulas} holds for trivial reasons.
\end{remark}

In Theorem \ref{explicit-theorem},
we solve the above recurrence relations and get
the following explicit formulas:
\[
N_{2n+1,0}\, =\,\frac14 N_{1,0}\, b^n\, (n+\frac12)^{n-2},\qquad  N_{2n+2,1}\, =\, N_{2,1}\,b^{n}\, (n+1)^{n-2},\qquad b=\frac{4N_{2,1}}{K^2}.
\]
As an immediate consequence of these expressions, we
evaluate the  growth rate
of the sequences
$N_{2n+1,0}$ and $N_{2n+2,1}$
and compare it with the growth rate of an analogous sequence of Gromov-Witten numbers (see Corollary \ref{growth}).

\begin{remark}
In the case of $X$ with $K^2=2$ and
$X_\R=\varnothing$, one can prove that $\CC_\R(m,k,x)=\varnothing$ for any $m, k$, and so, for such $X$, all
the numbers $N_{m,k}$ vanish.
This shows that Theorem \ref{main} can not be extended to this case.
\end{remark}

\subsection{Plan of the paper} Section 2 starts from constructing of natural, basic for our counting scheme, $\Pin^-$-structures and establishing their main properties.
We conclude this preliminary section by a discussion of wall-crossing and precise the limit behavior of the curves involved into counting of $N_{2m,1}$
in the case of contracting a spherical component of $X_\R$. Section 3 is devoted to a proof of the central result, Theorem \ref{main}.
Theorem \ref{magic-theorem} is proved in Section 4. In Section 5 we prove theorems \ref{two-recursive-formulas}, \ref{explicit-theorem} and
discuss a few simple
concrete applications. In Section 6 we are making few remarks on generating functions, on comparison of our count with a similar count for Gromov-Witten invariants,
and on a situation with other, $K^2>3$, del Pezzo surfaces.

\subsection{Acknowledgements} The idea of this work arose
during a stay of the second author at the Max-Planck Institute for Mathematics in Bonn in spring 2021 and took its shape during a RIP-stay of the authors
at the Mathematisches Forschungsinstitut Oberwolfach in summer 2021. We thank the both institutions for hospitality and excellent  (despite a complicated pandemic situation) working conditions.

Our special thanks go to R.~Rasdeaconu, discussions with whom were among the motivations for this study,
and to J.~Solomon, for encouragement and helpful remarks.

The second author was partially supported by the grant ANR-18-CE40-0009 of French Agence Nationale de Recherche
and the grant by Ministry of Science and Higher Education of Russia under the
contract 075-15-2019-1620 with St. Petersburg Department of Steklov Mathematical Institute.

\section{Preliminaries}\label{preliminaries}
From now on, we
denote the anticanonical degree $K^2$ by $d$.
\subsection{Basic $\Pin^-$ structures}\label{sect-basic}
Recall that each
del Pezzo surface $X$ of degree
$d=3$
has a natural anticanonical embedding $X\hookrightarrow\P^3$ representing it
as a non-singular cubic surface. For
del Pezzo surfaces of degree $d=2$,
the anticanonical map
is a double covering $X\to \P^2$ branched along a non-singular quartic curve, and its deck transformation $\gamma:X\to X$
is called {\it Geiser involution}. This covering
lifts naturally to an
embedding $X\hookrightarrow\P(1,1,1,2)$ into
the weighted projective space
$\P(1,1,1,2)$.  For
del Pezzo surfaces of degree $d=1$, the bi-anticanonical
map $X\to \P^3$ represents $X$ as a double covering of
a non-degenerate quadratic cone $Q\subset \P^3$ branched at its vertex
and along a non-singular sextic $C\subset Q$
(traced on $Q$ by a transversal cubic surface).
The deck transformation $\tau:X\to X$ in this case is known as the {\it Bertini involution}.
The above covering $X\to Q$ lifts to
an embedding  $X\hookrightarrow\P(1,1,2,3)$ provided by the graded anti-canonical ring $R=\sum_{m\ge 0} H^0(X;-mK)$.
In each case $d=1,2,3$, these embeddings of $X$ are defined uniquely up to
automorphisms of the ambient 3-space (see, for example, \cite{Dolgachev}).

These constructions are exhaustive, functorial, and work equally well over $\C$ and $\R$. In particular, over $\R$ we obtain a  natural embedding
of the real locus $X_\R$ into $\Rp3$ if $d=3$, into $\Rp2\times \R= \P_\R(1,1,1,2)\sm \v_q, \v_q=(0,0,0,1)$ if $d=2$, and
into $\Rp2\times \R= \P_\R(1,1,2,3)\sm \v_p, \v_p=(0,0,1,0)$ if $d=1$ (see \cite{TwoKinds}).

Therefore, in each case $d=1,2,3$, the real locus $X_\R$ inherits from the ambient $3$-space some natural $\Pin^-$-structures.
Namely, if $d=3$, we have a pair of
$\Spin$-structures in $\Rp3$ which differ by a shift by $h\in H^1(\Rp3;\Z/2)$, $h\ne0$.
They induce on $X_\R\subset\Rp3$ a pair of $\Pin^-$-structures, $\theta^X$ and $\theta^X+w_1$, where
 $w_1=w_1(X_\R)$ is the pull-back of $h$.
If $d=1$ or $2$, the above embeddings
$X_\R\subset\Rp2\times \R$ are two-sided, and thus a pair of $\Pin^-$-structures on
$\Rp2\times \R$ (which differ by a generator in $H^1(\Rp2\times \R;\Z/2)$) descends to a pair of $\Pin^-$-structures on $X_\R$,
also denoted $\theta^X$ and $\theta^X+w_1$ (since also differ by a shift by $w_1=w_1(X_\R)$).

Recall that there is a canonical correspondence between $\Pin^-$-structures $\theta$ on
$X_\R$ and {\it quadratic functions} $q_\theta: H_1(X_\R;\Z/2)\to\Z/4$,
that is the functions satisfying $q_\theta(x+y)=q_\theta(x)+q_\theta(y)+2(x,y)\mod4$.
These functions can be viewed as $\Z/4$-liftings of $w_1$
seen as a homomorphism $w_1:H_1(X;\Z/2)\to\Z/2$.
In particular, it implies that $q_{\theta+w_1}=-q_{\theta}$ (cf. \cite[Lemma 2.4.1]{TwoKinds}).

Thus, in the case $d=1$ or $3$, we distinguish the $\Pin^-$-structure $\theta^X$ from $\theta^X+w_1$ by requiring
that $\theta^X$ is {\it monic} that is
$q_{\theta^X}(w_1^*)=1$, where $w_1^*\in H_1(X_\R;\Z/2)$ is dual to $w_1$ (then $q_{\theta^X+w_1}=-q_{\theta^X}$ takes on $w_1^*$ value $-1$).
We call such $\theta^X$ {\it basic $\Pin^-$-structure} on $X_\R$.
In the case $d=2$, there is no natural way to distinguish $\theta^X$ from $\theta^X+w_1$ and we call {\it basic} both of them.
Note that our definition of basic structures is independent of the choice of a graded anticanonical embedding of $X_\R$.

As in \cite{TwoKinds}, we consider also the function
$$
\hat q_\theta : H_2^-(X)\to \Z/4,\quad
\hat q_\theta =q_\theta\circ \bhv
$$
where $H_2^-(X)=
\ker (1+\conj_* : H_2(X)\to H_2(X))$
 and $\bhv : H_2^-(X) \to H_1(X_\R;\Z/2)$ is
 the
 {\it Viro homomorphism} (see \cite[Chapter 1]{DIK} or \cite{TwoKinds}).
This homomorphism respects the intersection form, so that $\hat q_\theta$ inherits from $q_\theta$
the property
\[\hat q_\theta(x+y)=\hat q_\theta(x)+\hat q_\theta(y)+2(x,y)\mod4.\]

The Viro homomorphism $\bhv$ induces an isomorphism
\[H_1(X_\R;\Z/2)\cong H_2^-(X)/(1-\conj_*)H_2(X),\]
and has a simple interpretation
in differential topology setting.
Namely, for any
 real del Pezzo surface (and, more generally, for any compact complex surface
 $X$ with a real structure and $H_1(X;\Z/2)=0$) each class $\alpha\in H_2^-(X)$ can be realized by a $\conj$-invariant smoothly embedded oriented 2-manifold $F\subset X$
 and then it is $F\cap X_\R$, which is a collection (may be empty) of smooth circles in $X_\R$, that realizes the class $\bhv(\a)$. Furthermore, if $\bhv(\a)=0$ then $\a$ can be represented by a $\conj$-invariant smoothly embedded oriented 2-manifold disjoint from $X_\R$, and if $\bhv(\a)\ne 0$ then $F\subset X$ can be chosen in such a way
 that each of the components of $F\cap X_\R$ represents a non-zero element in  $H_1(X_\R;\Z/2)$.

 \begin{theorem}\label{basic}
 \,
\begin{enumerate}
\item\label{invar}
Real automorphisms and real deformations preserve the basic $\Pin^-$-structure $($resp. the pair of basic $\Pin^-$-structures$)$ of del Pezzo surfaces of degree 1 and 3 $($resp. of degree 2$).$
 \item\label{vanish} The quadratic functions $q_{\theta^X}$ of basic $\Pin^-$-structures $\theta^X$
 vanish on each real vanishing cycle in $H_1(X_\R;\Z/2)$.
 \end{enumerate}
 \end{theorem}
\begin{proof}
(\ref{invar}) holds because
real automorphisms and real deformations preserve $w_1$
and any real automorphism of $X$ is induced by a real automorphism of the ambient 3-space.
Property (\ref{vanish}) is a special case of \cite[Lemma 2.4.2]{TwoKinds}.
\end{proof}

\begin{proposition}\label{vanishing-in-degree-2}
For $d=2$ and any of the two basic $\Pin^-$-structures $\theta^X$, the following holds:
\begin{enumerate}
\item
The functions $\hat q$ and $q$ associated with $\theta^X$
are skew-symmetric with respect to $\gamma$ and its restriction $\gamma_\R:X_\R\to X_\R$, correspondingly. That is
$$q\circ(\gamma_\R)_*=-q,\qquad \hat q\circ(\gamma)_*=-\hat q.$$
\item For each pair $(m,k)\in \Z_{>0}\times\Z_{\ge 0}$, $k=m-1\mod 2$, the number $N_{m,k}$ does not depend on the choice of $\theta^X$. If, in addition, $m$ is odd, then
$N_{m,k}=0$.
\end{enumerate}
\end{proposition}

\begin{proof}
(1) The reflection map $\Rp2\times\R\to \Rp2\times\R$ given by
$(x,y,z,t)\mapsto (x,y,z,-t)$ shifts each of $\Pin^-$-structures on $\Rp2\times\R$ by $w_1(\Rp2\times\R)$ ({\it cf.} \cite[Lemma 1.10]{Kirby}).
Since $\gamma_\R$ is the restriction of this reflection to $X_\R$, it interchanges
the two basic $\Pin^-$-structures induced on $X_\R$, and hence transforms $q$ into $-q$.

Skew-symmetry of $\hat q$ follows from that of  $q$.

(2) Suppose first that $m=2n+1$.
Let $\mathbf x$ be a generic collection of
$k$ real points and $\frac12(m-k-1)$ pairs of
complex conjugate imaginary points and let $A$, $[A]\in \lay^m_\R(X)$,
be a  real rational curve passing through
$\mathbf x$.  Note that $A$ can not be $\gamma$-invariant,
since otherwise $\hat q([A])=-\hat q([\gamma (A)])=-\hat q([A])\in\Z/4$,
which implies that $\hat q([A])$ is even and thus
contradicts to the congruences
$\hat q([A])\equiv[A]^2\equiv -AK\equiv 2n+1\mod 2$.
Therefore,
the set of such curves $A$ splits into a finite union of pairs, $\{A, \gamma (A)\}$, and there remains to notice that
$i^{\hat q([A])-1}+ i^{\hat q([\gamma (A)])-1}=0 $ due to
$\hat q([A])+\hat q([\gamma (A)])=0$, $\hat q([A])=\pm 1$.

In the case $m=2n$, it is sufficient to notice that then $\hat q(\a)=-\hat q(\a)\in\Z/4$ for any $\alpha\in\lay^m_\R$, since (similar to the above)
$\hat q(\a)\equiv \a^2\equiv -\a K\equiv 2n\mod 2$.
\end{proof}

\subsection{Three auxiliary surfaces}\label{3-surfaces}
To simplify proving Theorem \ref{two-recursive-formulas} (see Section \ref{proof-recursion})
we pick some particular real del Pezzo surface $X$ for each of the degrees $d=1,2,3$.

 \begin{proposition}\label{aux-surfaces}\,
 \begin{enumerate}[leftmargin=*,align=parleft,  labelsep=2mm,]\item[$(1)$] For $d=1$, pick $X$ with $X_\R=\Rp2\dsum 3\SB^2$. Then $:$
\newline
$\bullet$ $H_2^-(X)\cong\Z^2$ is generated by $K$ and a real root vector $e\in K^\perp$,
while the first real layer $\lay^1_\R(X)$ consists of $-K$ and the divisor classes of 2 real lines, $-K\pm e$.
\newline
$\bullet$
The Bertini involution acts on  $H_2^-(X)$ preserving the basic quadratic form $\hat q$ and $-K$, but permuting the divisor classes of lines.
\item[$(2)$]
For $d=2$, pick $X$ with $X_\R=\K\dsum \SB^2$ {\rm(}$\K$ denotes a Klein bottle{\rm ).}
Then $:$
\newline
$\bullet $ $H_2^-(X)\cong\Z^4$ is generated by $K$, three pairwise orthogonal real root vectors $e_1,e_2,e_3\in K^\perp$, and $\frac12(-K-e_1-e_2-e_3)$,
while the first real layer $\lay^1_\R(X)$ consists of divisor classes of $8$ real lines,
$\frac12(-K\pm e_1\pm e_2\pm e_3)$.
\newline
$\bullet$
$X$ can be chosen so that its group of automorphisms contains $\Z/2\times\D4_4$, where $\Z/2$ is generated by the Geiser involution $\gamma$
and $\D4_4$ is the dihedral group of a square.
\newline
$\bullet$
 $\Z/2\times\D4_4$ acts transitively on the above $8$ lines, and, in particular, one of the $\D4_4$-orbits
$\{\frac12(-K-e_1-e_2-e_3),\frac12(-K+e_1+e_2-e_3),\frac12(-K+e_1-e_2+e_3),\frac12(-K-e_1+e_2+e_3)\}$ is permuted by $\gamma$-action
with the other $\D4_4$-orbit.
The $\D4_4$-action preserves each of the basic quadratic functions $\pm \hat q$, while $\gamma_*$ acts as $(-1)^n$
on $\hat q|_{\lay^n_\R}$.
\newline
$\bullet$
The multiples of $K$ are the only $\gamma_*$-invariant
classes
and the only ones invariant under the $\D4_4$-action.
\item[$(3)$]
 For $d=3$, pick $X$ with  $X_\R=\Rp2\dsum \SB^2$. Then  $:$
\newline
$\bullet$
$H_2^-(X)$ is generated by the divisor classes of
the 3 real lines $L_1, L_2, L_3\subset X$.
\newline
$\bullet$
$X$ can be chosen so that its automorphism group contains the symmetric group $S_3$
acting by permutations on the above lines.
\newline
$\bullet$
The $S_3$-action in $H_2^-(X)$ preserves the basic quadratic function $\hat q$.
\newline
$\bullet$
The multiples of $K$ are the only classes invariant under this action.
\end{enumerate}
 \end{proposition}

 \begin{proof}
 In all 3 cases we use
 the
 1--1 correspondence
between real lines and divisor classes $\a \in \lay^1_\R$ with $\a^2=-1$ and apply
the Lefschetz fixed point theorem, $\rk H_2^+(X)-\rk H_2^-(X)=\chi(X_\R)-2$, which together with $\rk H_2^+(X)+\rk H_2^-(X)=\rk H_2(X)=10-d$
calculates
the ranks of $H_2^\pm(X)=\ker(1\mp\conj_*)$.
We use also that
the lattices
$H^\pm_2(X)=\ker(1\mp\conj_*)$ have 2-periodic discriminant groups of rank
$\frac12(10-d-\rk H^*(X_\R))$ (see, f.e., \cite[Proposition 8.3.3]{DIK}).
Recall besides that both Bertini (for $d=1$) and Geiser
(for $d=2$) involutions act on $H_2(X)$ as a reflection against the line
spanned by $K$.

(1) In this case, $K^\perp\cap H_2^-\cong\langle -2\rangle$.
The divisor classes of real lines  $\a\in \lay^1_\R$ split as $\a=-K\pm e$, where $e$ is a generator of $\langle -2\rangle$,
and the Bertini involution permutes
these lines.
Preserving of $\hat q$ by its action is due to Theorem \ref{basic}(1).

(2) In this case, $\rk H_2^\pm=4$ and $\rk (K^\perp\cap H_2^-)=3$.
Since $X$ is an $(M-2)$-surface of type $I$,
the discriminant form of $H_2^+$
is even 2-periodic of rank 2.
Thus, we conclude that $H_2^+\cong D_4$.
By Nikulin's gluing theorem \cite{Nik},
the discriminant group $\discr(K^\perp\cap H_2^-)$
is also 2-periodic and has rank $1$ or $3$ (since $K^\perp\cap H_2^-$ is complementary to $H_2^+$ in $K^\perp=E_7$).
In the case of discriminant rank 1, the Brown invariant would be
$\pm 1\in\Z/8$, which
contradicts to Brown's congruence
${\rm Br}(\discr(K^\perp\cap H_2^-))=\sigma(K^\perp\cap H_2^-)= -3\mod 8$.
Thus, the discriminant rank is 3, which implies $K^\perp\cap H_2^-\cong 3A_1$.
Arithmetical description of the real lines is then straightforward.

For constructing $X$ with the required $\D4_4$-symmetry, it is sufficient to pick a $\D4_4$-symmetric perturbation of a $\D4_4$-symmetric pair of conics, as shown on Fig. \ref{4lines}.
\begin{figure}[h!]
\center
\caption{Degree 2 del Pezzo surface with $\D4_4$-symmetry lifted to
a double plane
branched along a $\D4_4$-symmetric quartic curve
}\label{4lines}
\includegraphics[width=0.4\textwidth]{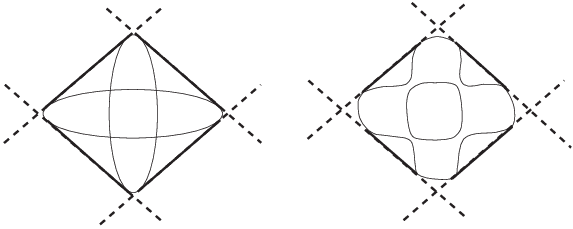}\\
{\small \sl The thick/dotted coloring of bitangents
shows how they are lifted so that
$\D4_4$-symmetry is preserved in the double covering.}
\end{figure}
Note that the above $\D4_4$-action is realized by projective transformations $\P^2\to \P^2$,
and hence has a "cylindrical" lifting to $\P(1,1,1,2)$ which preserves $X$. Such transformations $\P(1,1,1,2)
\to \P(1,1,1,2)$ preserve each of the two $\Pin^-$-structures on $\P_\R(1,1,1,2)$
and send outward vector fields on $X_\R$ to outward fields. Therefore, each of $\D4_4$-symmetries preserves each of the basic $\Pin^-$-structures on $X_\R$
and thus preserves $\hat q$. Proposition \ref{vanishing-in-degree-2} gives also
the required skew-symmetry of the $\gamma$-action on $\hat q$.

(3) In this case,
$\rk H^-_2(X)=3$ and its discriminant group has order 4.
Since the intersection matrix
\scalebox{0.6}{$\begin{bmatrix}-1&1&1\\1&-1&1\\1&1&-1\end{bmatrix}$}
of the three real lines lying on $X$ (the only existing ones)
has the same determinant $4$,
these lines must generate the whole lattice $H^-_2(X)$.

An example of $X$ with a required $S_3$-action can be
given by an equation of the form $x_0(x_1^2+x_2^2+x_3^2-x_0^2)=\e x_1x_2x_3$.
This action preserves $\hat q$ due to Theorem \ref{basic}(1).
The last claim about invariant classes
follows from linear independence of divisor classes of the lines $L_1,L_2,L_3$ and the relation $[L_1]+[L_2]+[L_3]=-K$.
\end{proof}

\subsection{Wall-crossing}\label{surgery}
We call a complex analytic family $X(z)$,  $z\in \mathbb D=\{z\in\C\,|\, \vert z\vert <1\}$
 a {\it Morse-Lefschetz family}, if :
$\X=\cup_z X(z)$ is a non-singular 3-fold;
each $X(z)$ with $z\ne 0$ is non-singular
while $X(0)$ is uninodal; the projection $\pi: \X\to \mathbb D$ is a proper map
and it is a submersion at every point except the nodal point of  $X(0)$; at the nodal point of $X(0)$ in appropriate local coordinates $z_1,z_2,z_3$ the projection can be written as $z= \sum_i z_i^2$.
If $\X$ is equipped with a real structure sending $X(z)$ to $X(\bar z)$ for each $z\in \mathbb D$, the family is called {\it real}.

The following
is well known (for del Pezzo surfaces with $K^2=1$, see, for example, \cite{Combined}).

\begin{proposition}\label{connecting}
Any pair of real del Pezzo surfaces of same degree $d\le 3$
can be connected either by a real deformation, or
a finite sequence of real Morse-Lefschetz families of del Pezzo surfaces embedded into the corresponding projective or weighted projective space.
If the real locus of the both surfaces is non-empty the sequence of real Morse-Lefschetz families can be chosen not to involve surfaces with empty real locus.
\qed\end{proposition}

Every
real Morse-Lefschetz family $\pi : \X\to \mathbb D$ being restricted to the punctured upper half disc $\mathbb D^+=\{z\in \mathbb D, \Im z\ge 0, z\ne 0\}$ and considered as
a fibration of smooth $4$-manifolds is trivial. This provides a natural identification of lattices $H_2(X(t)), t\in \R^*, \vert t\vert <1$, so that for these values of $t$ the involutions $c_t:
H_2(X(t))\to H_2(X(t))$ and $c_{-t}: H_2(X(-t))\to H_2(X(-t))$ induced by the complex conjugation
become related by the Picard-Lefschetz transformation:
\begin{equation}
c_{t}=s_e\circ c_{-t} \quad s_e(v)=v+(ev)e,
\end{equation}
where $e\in H_2(X(t))$ is the vanishing class of the underlying nodal degeneration.

\begin{proposition}\label{viro-hom} Let $\pi : \X\to \mathbb D$ be a real Morse-Lefschetz family of del Pezzo surfaces of degree $K^2\le 3$,
embedded into the corresponding projective or weighted projective space,
and let the direction of the wall-crossing is selected
so that $\chi(X_\R(-t))<\chi (X_\R(t))$ for $t>0$.
Then, $H_2^-(X(t))$ is naturally identified with $H_2^-(X(-t))\cap e^\perp$ and with
respect to this identification $\hat q_{\theta^{X(t)}}$ is the restriction of $\hat q_{\theta^{X(-t)}}$, where in the case $K^2=2$ we pick arbitrary $\Pin^-$-structure
on the ambient $\P(1,1,1,2)$ and choose that basic $\Pin^-$-structures which are induced by means of outward vector fields.
\end{proposition}
\begin{proof} To identify the lattices $H_2(X(t)), t\in \R^*, \vert t\vert <1$ we use, as it is described above, the trivialization of the family
over the upper disc. Then, under assumption $\chi(X_\R(-t))<\chi (X_\R(t))$ for $t>0$, the identity $c_{t}=s_e\circ c_{-t}$ implies that $e\in H_2^-(X(-t))$ and $H_2^-(X(t))
=H_2^-(X(-t))\cap e^\perp$. Also, $e\in H_2^+(X(t))=\ker (1-\conj_*)$. Next, we follow the geometric interpretation of the Viro homomorphism $\bhv$
(see Section \ref{sect-basic}).

We fix $t_0>0$ and pick a class $\a\in H_2^-(X(t_0))$. If $\bhv\a\ne 0$,
then we realize it by
a $\conj$-invariant smooth embedded oriented 2-manifold $F(t_0)\subset X(t_0)$ disjoint from $X_\R(t_0)$. If $\bhv(\a)\ne 0$, we choose $F(t_0)\subset X(t_0)$ in such a way
 that each of the components of $F(t_0)\cap X_\R(t_0)$ represents a non-zero element in  $H_1(X_\R(t_0);\Z/2)$. In both cases, if one of spherical components of $X_\R(t_0)$
 is a vanishing sphere $S^2(t_0)$, then $F(t_0)$ does not intersect $S^2(t_0)$ (since the circle components of intersection should be homologically non-trivial, while
a presence of an isolated intersection point contradicts to smoothness of $F(t_0)$ and its $\conj$-invariance). If a vanishing sphere $S^2(t_0)$, chosen $\conj$-equivariant, is not contained in $X_\R(t_0)$,
then, due to our assumption on the direction of the wall crossing, $S^2(t_0)\cap X_\R(t_0)$ consists of 2 points, while the intersection number $e\cdot  \a$ is $0$
(since $e$ and $\a$ belong to opposite eighen-spaces for the action of $\conj$ in $H_2(X(t_0))$).
Thus, we can eliminate, respecting $\conj$-invariance, all geometric intersection points between $F(t_0)$ and $S^2(t_0)$ that might appear when  $S^2(t_0)\cap X_\R(t_0)\ne\varnothing$, since in such a case any 2 complex conjugate intersection points are of opposite intersection index and, hence, can be eliminated by a surgery of $F(t_0)$ along a $\conj$-invariant path joining them along $S^2(t_0)$.

As soon as $F(t_0)\cap S^2(t_0)=\varnothing $, there is no more obstruction for extending $F(t_0)$ up to a locally trivial $\conj$-invariant family
$F(t)\subset X(t), t\in\mathbb D_\R,$ avoiding the node of $X(0)$. Now, the invariance of $q_{\theta^{X_\R(t)}}([F(t)\cap X_\R(t)])$ follows from its continuity
(note that $q^{X_\R(0)}$ is well defined on $H_1$ of the smooth part of $X_\R(0)$).
\end{proof}

Recall that {\it untwisting} of a given real Morse-Lefschetz family $X(z)$ induced by
the substitution $z=t^2$ (see \cite{Combined})
and followed by resolution of the nodal singularity defines a complex analytic fibered 3-fold
$\hat\pi:\hat\X\to \mathbb D$
equipped with a real structure.
This new family $\hat\X(t)=\hat\pi^{-1}(t), t\in \mathbb D$ is formed by
$\hat\X(t)=X(t^2)$ for $t\ne0$ and a reducible fiber
$\hat\X(0)=\hat X(0)\cup\Cal E$,
where $\hat X(0)$ is the resolution of the nodal surface $X(0)$,
$\Cal E$ is a non-singular projective quadric surface,
and their intersection $E=X(0)\cap\Cal E$
can be considered at the same time as the exceptional divisor of the blow-down $\hat X(0)\to X(0)$ and as a hyperplane section of $\Cal E$ viewed as a projective quadric.

The following lemma, which will be used in Section \ref{computation-section},
requires a more flexible symplectic setting.
So, we endow $\X$
and then $\hat\X$ with a K\"ahler structure $\omega$ respected by the complex conjugation $\conj$
and then take a small real perturbation of
the (almost) complex structure in $\hat\X$ to make it generic in the class of
 {\it admissible almost-complex structures} on a {\it symplectic fibration} $(\hat\pi:\hat\X\to \mathbb D, \omega)$ (see \cite{Teh} for definitions and existence).

 In this lemma we consider a very special type of wall-crossing, {\it contraction of a spherical component}, meaning by this that
 the node of $X_\R(0)$ is solitary and give birth to an $S^2$-component in $X_\R(t)$, $t>0$, or equivalently that $\Cal E_\R=S^2$
 and $E_\R=\varnothing$.

  \begin{lemma}\label{bijection}
 Let $\hat\pi : \hat{\X}\to \mathbb D$ be a described above real symplectic fibration
$($equipped with an admissible generic almost-complex structure$)$
 contracting a spherical component, and let $\mathbf x(t)= (x(t), z_1(t), w_1(t), \dots, z_{m-1}(t), w_{m-1}(t))$ be
a generic $\conj$-equivariant family of point-constraints  such that $x(0)$ is real and belongs to the
sphere $S^2 = \Cal E_\R$,  while $w_1(0)=\conj z_1(0), \dots, w_{m-1}(0)= \conj z_{m-1}(0)$ are imaginary and belong to $\hat X(0)$. Then:
 \begin{enumerate}
 \item For any selection $p_1\in\{z_1(0),w_1(0)\}, \dots,  p_{m-1}\in\{z_{m-1}(0),w_{m-1}(0)\}$, each of $J$-holomorphic rational curves $C\subset \hat X(0)$
 having $-K\!\!\cdot\! C=m$ and passing through $p_1,\dots, p_{m-1}$ is irreducible nodal and intersects $E$ transversally. The number of such curves $C$ in a given divisor class
 does not depend on the selection of points $p_1,\dots, p_{m-1}$.
 \item For any continuous family $A(t)\in \Cal C_\R(2m,1,\mathbf x(t))$,
 the limit-curve $\lim_{t\to 0^+} A(t)$ is a curve-configuration
 $$C\cup \conj C\cup L_1\cup \conj L_1\cup\dots\cup L_{n-1}\cup \conj L_{n-1}\cup H_n\subset\hat\X(0)$$
 with the following properties: $C$ is a $J$-holomorphic rational curve which lies in $\hat X(0)$ and passes through a selection $p_1\in\{z_1(0),w_1(0)\},$
 $ \dots,$  $p_{m-1}\in\{z_{m-1}(0),w_{m-1}(0)\}$; $-K\cdot C=m$; $n=C\cdot E$;
 $C\cap E$ consists of $n$ distinct points $q_1,\dots, q_n$; $L_1,\dots, L_{n-1}$ are $\P^1$-generators of $\Cal E\cong \P^1\times \P^1$ and pass through
 a selection of $n-1$ points between $q_1,\dots, q_n$; $H_n$ is a real hyperplane section of $\Cal E\cong \P^1\times \P^1$ that passes through the remaining point and the point $x(0)$.
 \item For any sufficiently small $t>0$, passing to the limits as in item (2) establishes a bijection between the set  $\Cal C_\R(2m,1,\mathbf x(t))$
 and the set $I$ of curve-configurations described in item (2). The set $I$ is a disjoint union of $2^{m-2}$ subsets $I_{\{\mathbf p, \conj  \mathbf p\}}$ where $\{\mathbf p,
 \conj \mathbf p\}$ is pair of complex conjugate selections
 and $I_{\{ \mathbf p, \conj  \mathbf p\}}$ consists of configurations
 with $\mathbf p\subset C$, $\conj \mathbf p\subset \conj C$.
 Each $I_{\{\mathbf p, \conj\mathbf p\}}$ is itself a disjoint union of subsets $I_{\{C, \conj C\}}$ consisting of configurations sharing the
 same pair $\{C, \conj C\}$. Each $I_{\{C,\conj C\}}$ contains $n2^{n-1}$ elements.
\end{enumerate}
 \end{lemma}

 \begin{proof} Claim (2) and the first part of Claim (1) follow from \cite[Proposition 3.7]{Br-P}
 (the option of tangency between $H$ and $E$ is excluded, since $E_\R=\emptyset$). Due to the transversality of curves $C$ in Claim (1), their number is nothing but
a relative Gromov-Witten invariant of pair $(\hat X(0), E)$ (without constraints on $E$).
 The bijectivity in Claim (3)
 follows from Symplectic Sum Formula in its simplest, transversal intersection, case. The
 rest of Claim (3) is
 a straightforward consequence of bijectivity.
 \end{proof}

\section{Proof of Theorem \ref{main}}
\label{proofs-section}
In this section we consider real del Pezzo surfaces $X$ of degree $d\le3$,
with a basic $\Pin^-$-structure $\theta^X$ (cf. Theorem \ref{basic})
and the associated with it quadratic
function $\hat q_{\theta^X}: H_2^-(X)\to\Z/4$. In Introduction
we omitted, for shortness, indicating $X$ and $\theta^X$.
In this section we
need to use a full notation, since $X$ will be varying.
So, let us adjust our notation and rewrite our main definition (\ref{main-definition}) correspondingly:
$$
N^X_{m,k}=\sum\limits_{\a\in\lay^m_\R(X)}\W^X_{\a,k}, \,\, \W^X_{\a,k}= i^{\hat q_{\theta^X}(\a)-m^2}W^X_{\a,k},
 \,\, W^X_{\alpha,k}=\hskip-6mm\sum\limits_{A\in \Cal C_\R(m,k,\bold x), [A]=\alpha} \hskip-5mm w(A).
$$

Recall that the {\it traditional Welschinger invariants} are defined similarly as
$$\til W^{X}_{\alpha,k}=\hskip-5mm\sum\limits_{A\in \Cal C_\R(m,k,\bold x), [A]=\alpha} \hskip-3mm\til w(A),\quad
\text{ with }\quad\til w(A)=(-1)^{s_A}
$$
where $s_A$ stands for the number of solitary real nodal points of $A$.
Since $w(A)=(-1)^{c_A}$ and $c_A+s_A$ is the arithmetic genus $g_a(\a)$, where $\a=[A]$ (and $c_A$ is the number of cross-point real nodes), we have
$$\til W^{X}_{\a,k}=(-1)^{g_a(\a)}\,W^{X}_{\a,k}.$$

\subsection{Push-forward formula}\label{push}
Assume that real del Pezzo surfaces $X$ and $Y$ are related by wall-crossing,
and consider a real Morse-Lefschetz family $\pi : \X\to \mathbb D$ such that
$X=X(-t_0)$ and $Y=X(t_0)$ for a fixed $0<t_0<1$. Following the exposition in Section \ref{surgery}, we identify
the lattices $H_2(X)=H_2(Y)$ using a smooth trivialization of the fibration $\pi : \X\to \mathbb D$ over the upper half disc.

We choose the direction of wall-crossing so that $\chi(X_\R)<\chi(Y_\R)$, which implies
$$
H_2^-(Y)=\{\alpha\in H_2^-(X)\,|\, \alpha e=0\} \quad \text{where $e\in H_2^-(X) $ is the vanishing class},
$$
and consider the orthogonal projection
\begin{equation}\label{projection}
p^e: H_2^-(X)\to H_2^-(Y)\otimes\Q
%\frac12H_2^-(Y)
\qquad v\mapsto v+\frac12(ev)e.
\end{equation}
Our aim is to treat the push-forward by $p^e$ of the function $\W_k^X: H_2^-(X)\to \Z$
defined by $\W_k^X(\a)=\W^X_{\a,k}$ if $\a$ is an effective divisor class
 and $\W_k^X(\a)=0$ otherwise.

Note that, due to Theorem \ref{basic}(1) and Proposition \ref{vanishing-in-degree-2}(2), none of the above ingredients, including $p^e$ and $\W_k^X$, depends on a choice of $t_0$.

\begin{proposition}\label{push-forward}
If $X_\R$ and $Y_\R$ are non-empty, then:
{\rm (1)} the push-forward $p^e_* (\W_{k}^X)$ is well-defined, {\rm (2)}
it is supported in the integer part
$H_2^-(Y)=H_2^-(Y)\otimes1\subset H_2^-(Y)\otimes\Q$,
and {\rm (3)} being restricted to this integer part,  $p^e_* (\W_{k}^X)=\W_{k}^Y$ on $H_2^-(Y)$.
\end{proposition}

\begin{proof} By definition,
$p^e_*(\W_{k}^X )(\alpha)=\sum\limits_{n\in \Z} \W_{k}^X(\alpha + n\,e)$,
and (1) means finiteness of the sum.
But if $\vert n\vert$ is sufficiently large, then
$(\alpha + ne)^2 < -2$ and thus,
$\alpha + ne$ can not be realized
by a reduced irreducible rational curve, which gives
$\W_{k}^X(\alpha + ne)=0$.

According to the projection formula (\ref{projection}), $p^e_* (\W_{k}^X)$ is supported in
$H_2^-(Y)\otimes \frac12\Z$,
and (2) means that
\begin{equation*}
\sum_{n\in \Z}\W_{k}^X(\alpha + ne)=0 \quad\text{ for each }\ \alpha\in H_2^-(X) \text{ with }
r=\alpha e\ \text{ odd}.
\end{equation*}
To prove that this sum is zero, we group the summands by pairs $\alpha+ne$, $\alpha+(r-n)e$ and notice
that the reflection $s_e$
permutes the elements in each of these pairs. As is known (see for example \cite[Theorem 4.3(1)]{IKS})
the  traditional Welschinger invariants
are preserved under such a reflection, so that $\til W^{X}_{\alpha+ne,k}=\til W^{X}_{\alpha+(r-n)e,k}$
As it follows from adjunction formula, the curves in the divisor class $\alpha+ne$
and the curves in the divisor class $\alpha+(d-n)e$ have the same parity of the arithmetic genus.
Hence, $W^{X}_{\alpha+ne,k}$ and $W^{X}_{\alpha+(r-n)e,k}$ are also equal. Thus, there remains to notice, that
$(\alpha+ne)K=\alpha K= (\alpha+(r-n)e)K$, and
that
\begin{equation*}
\begin{aligned}
i^{\hat q_{\theta^X}(\alpha+(r-n)e)}=&i^{\hat q_{\theta^X}(\alpha+ne)+\hat q_{\theta^X}((r-2n)e)+2(\alpha+ne)((r-2n)e))}\\
&=i^{\hat q_{\theta^X}(\alpha+ne)+2r(r-2n)}=-i^{\hat q_{\theta^X}(\alpha+ne)}
\end{aligned}
\end{equation*}
due to $\hat q_{\theta^X}(e)=0$ (see Theorem \ref{basic}) and $r=1\mod 2$.

To prove (3) we apply Theorem 2.1 from \cite{Br}.
According to this theorem, if $\a e=0$, then
$$
\til W^{Y}_{\a,k}=\sum_{n\in \Z}(-1)^n\til W^{X}_{\a+n\,e,k}.
$$
This implies
$$
W^{Y}_{\alpha,k}=\sum_{n\in \Z} W^{X}_{\alpha+ne,k},
$$
since, due to the adjunction formula, $g_a(\a + ne)=g_a(\a)+n\mod2$ as soon as $\a e=0$.
Finally, it is left to notice that
$\hat q_{\theta^X}(\a+ne)=\hat q_{\theta^X}(\a)$ for all $n\in \Z$,
and that $\hat q_{\theta^Y}(\a)=\hat q_{\theta^X}(\a)$ due to Proposition \ref{viro-hom}.
\end{proof}

\subsection{Proof of Theorem \ref{main}}
Due to deformation invariance of $\hat q$ (see Theorem \ref{basic}), and due to independence of $\W_k^X$ from a choice of a basic $\Pin^-$-structure when $K^2=2$ (see Proposition \ref{vanishing-in-degree-2}(2)), it is left to consider sequences of real Morse-Lefschetz families as in Proposition \ref{connecting} and to apply Proposition \ref{push-forward} (noticing that the projection $p^e$ preserves  the layers).

\section{Proof of Theorem \ref{magic-theorem}}
\label{computation-section}
Throughout this section
we fix a  real del Pezzo surface $X$ of degree $d=K^2\le3$ and use notation
$\hat q: H_2(X)\to\Z/4$
for the quadratic function associated to a basic $\Pin^-$-structure on $X_\R$.

\subsection{Switch to open Gromov-Witten invariants}\label{switch}
In the computations below we essentially rely on Solomon's recursion relations for open Gromov-Witten invariants (see \cite{Jake}).
In accordance with notation in \cite{Jake},
for each $m\ge 1$, $0\le k\le m-1, k= m-1 \mod 2,$ and each $v\in K^\perp\cap \ker (1+\conj_*)$,
we put
\begin{equation}\label{individual}
\begin{aligned}
&N_{m,v, k}=\sum_{{\substack{ A\in \Cal C_\R(m,k,\bold x)\\ d[A]=-mK-v}}}i^{\hat q([A])-m^2}w(A), \\
&\Gamma_{m,v, k}= -2^{1-l} \sum_{{\substack{ A\in \Cal C_\R(m,k,\bold x)\\ d[A]=-mK-v}}}i^{\hat q([A])-m^2}w(A),
& \quad  l=\frac12(m-k-1)
\end{aligned}
\end{equation}
where $\mathbf x$ is a generic collection of $k$ real points and $l=\frac12(m-k-1)$ pairs of
complex conjugate imaginary points,
and define
$$
\Gamma_{m, k}= \sum_{v\in K^\perp\cap \ker (1+\conj_*)} \Gamma_{m,v, k}.
$$
Note that:
\begin{itemize}
\item If $-mK-v$ is not divisible by $d$, then the sums involved in (\ref{individual}) are void and the numbers $N_{m,v, k}, \Gamma_{m,v, k}$ are zero by definition.
\item Numbers $N_{m,v, k}$ and $\Gamma_{m,v, k}$ does not depend on a choice of $\mathbf x$ and are preserved under equivariant deformations and equivariant isomorphisms ({\it cf.} discussion in Introduction).
\item
In accordance with (\ref{main-definition}) and above definitions we have
\begin{equation}\label{GammaToN}
N_{m,k}= -2^{l-1} \sum_{v\in K^\perp\cap \ker (1+\conj_*)} \Gamma_{m,v, k} = -2^{l-1}\Gamma_{m,k},
\end{equation}
while Theorem \ref{main} implies that both $N_{m,k}$ and $\Gamma_{m,k}$ do not depend even on a choice of a real del Pezzo surface of given degree.
\item If $d$ is odd, we have
$$
\hat q([A])-m^2= \hat q(d [A]) - m^2= \hat q(-mK-v)-m^2= \hat q(v),
$$
so that the above definitions can be rewritten as follows
\begin{equation*}
\begin{aligned}
N_{m,v, k}=&\sum_{{\substack{ A\in \Cal C_\R(m,k,\bold x)\\ K^2[A]=-mK-v}}} i^{\hat q(v)}w(A), \\
&\Gamma_{m,v, k}= -2^{1-l} \sum_{{\substack{ A\in \Cal C_\R(m,k,\bold x)\\ d[A]=-mK-v}}} i^{\hat q(v)}w(A)
& \quad  l=\frac12(m-k-1)
\end{aligned}
\end{equation*}
\item
Theorem \ref{two-recursive-formulas} is reformulated as
\begin{equation}\label{recurrence-via-Gamma}
\small
\begin{aligned}
2md\,\Gamma_{m,0}+
&\sum_{j=1}^{n}\binom{n-1}{n-j}j(m-2j)^2\Gamma_{m-2j,0}\Gamma_{2j,1} \quad\text{for } m=2n+1,\\
2md\,\Gamma_{m,1}+
&\sum_{j=1}^{n}\binom{n-1}{n-j}j(m-2j)^2\Gamma_{m-2j,1}\Gamma_{2j,1} \quad\text{for }  m=2n+2.
\end{aligned}
\end{equation}

\end{itemize}

\subsection{Proof of Theorem \ref{magic-theorem}} \label{magic}

\begin{lemma}\label{binom-sum} For any $n\in \N$, we have\quad
$n 2^n =\sum\limits_{k=0}^n (n-2k)^2\binom{n}k$.
\end{lemma}
\begin{proof} The identities $\sum_{k=0}^n\binom{n}k=2^n$, $\sum_{k=0}^n k \binom{n}k=n2^{n-1}$, and $\sum_{k=0}^n k^2\binom{n}k=(n+n^2)2^{n-2}$ imply
$\sum_{k=0}^n (n-2k)^2\binom{n}k= n^2 2^n- 4n^2 2^{n-1}+ 4(n+n^2)2^{n-2}= n2^n.$
\end{proof}

\begin{proof}[Proof of Theorem \ref{magic-theorem}]
Due to Theorem \ref{main}, it is enough to check the formula (\ref{mag-formula})
for one particular del Pezzo surface $X$ of each degree $d\le3$.
We make choice by picking  $X=X(t_0)$, $0<t_0\ll 1$, from
 a real Morse-Lefschetz family $X(t)$ contracting a spherical component $S^2\subset X_\R(t_0)$
(see Sec. \ref{surgery}).

First,
we interpret the right-hand side of (\ref{mag-formula})
as a weighted count of complex rational curves
$C$ on the resolution $\hat X(0)$ of the nodal surface $X(0)$ that (1) belong to the $m$-th level  $\lay^m(\hat X(0))$, (2) pass through a fixed
generic collection $\mathbf p$ of $m-1$ points, and (3) have a non-trivial intersection, $C\!\cdot\! E>0$,  with the $(-2)$-curve $E\subset \hat X(0)$
representing the node. Namely, we observe that, in accordance with the
 Abramovich-Bertram-Vakil formula (see \cite[Proposition 4.1]{IKS} and \cite[Theorem 2.5]{Br-P}),
the input of each of such curves $C$ into the right-hand side is equal to $\sum_{k=0}^n (n-2k)^2\binom{n}k$ where
$n=
C\!\cdot\! E$.
Thus, in accordance with Lemma \ref{binom-sum},  the right-hand side of (\ref{mag-formula})
can be seen as the weighted count of the above curves $C\subset \hat X(0)$
with weights $n2^n$.

To treat the left-hand side we consider the untwisted family  $\hat\X(t)$ and apply Lemma \ref{bijection}
choosing the constraint  $\mathbf x(t)= (x(t), z_1(t), w_1(t), \dots, z_{m-1}(t),$ $ w_{m-1}(t))$ as indicated there.
First, we note that the input of each of the curves $A(t_0)\in\CC(2m,1,\mathbf x(t_0))$
into $N_{2m,1}$ is equal to $1$.
This is because $\hat q$ vanishes on the spherical component of $X_\R(t_0)$ containing $x(t_0)$
and curves $A_\R(t)$ have no real
cross-point nodes for all sufficiently small $t>0$ (the latter follows from
the explicit description of the limit curves $\hat A_\R(0)$ in Lemma \ref{bijection}(2)).
So, $N_{2m,1}$ is just the cardinality of  $\CC(2m,1,\mathbf x(t_0))$, which in its turn coincides with the cardinality of the set
$I$ that counts the limit curve-configurations see Lemma \ref{bijection}(3).

To compare this cardinal count of the limit curve-configurations with the weighted count we made at the beginning, let us choose as $\mathbf p$ a selection
$p_1\in\{z_1(0),w_1(0)\},$
 $ \dots,$  $p_{m-1}\in\{z_{m-1}(0),w_{m-1}(0)\}$.
Restricting the cardinal count to any particular
selection $\mathbf p$
is equivalent to dividing $N_{2m,1}$ by $2^{m-2}$
(in accord with subsets $I_{\{\mathbf p, \conj \mathbf p\}}$ in Lemma \ref{bijection}(3)).
On the other hand, forgetting the line components
in the curve configuration is equivalent to counting with the weight $n2^{n-1}$. Thus, we conclude that
\begin{equation*}
\sum_{\alpha\in\lay^m} (e\alpha)^2 GW_{\alpha}=2\cdot 2^{-(m-2)}N_{2m,1}=2^{3-m}N_{2m,1}. \qedhere
\end{equation*}
\end{proof}

Due to (\ref{GammaToN}), the result of Theorem \ref{magic-theorem} can be rewritten as follows.
\begin{cor}\label{miracle}  In the same setting as in Theorem \ref{magic-theorem}, we have
\begin{equation}\label{withGamma}
\pushQED{\qed}
- 2 \Gamma_{2m,1}
= \sum_{\alpha\in\lay^m} (e\alpha)^2 GW_{\alpha}\, .
\qedhere\popQED
\end{equation}
\end{cor}

\section{Proof of Theorem \ref{two-recursive-formulas}}\label{proof-recursion}
\subsection{Preparation for proving Theorem \ref{two-recursive-formulas}}
Apart of Theorem \ref{main} and Proposition \ref{miracle},
the proof is based
on the following Solomon's recursion rule ({\it cf.} \cite[(OGW3)]{Jake} and \cite[Th.1.1(RWDVV3)]{Chen})
which is a corollary of an analog for WDVV-equation designed by Solomon  \cite{open} in the framework of open strings.

\begin{theorem}\label{ThirdRelation}\footnote{In fact, this theorem, with appropriate definitions for the numbers $\Gamma_{B,k}$, holds for all real rational surfaces $X$ with any $\Pin^-$-structure on $X_\R$.} Let $X$ be a real del Pezzo surface equipped with a basic $\Pin^-$-structure, and
let $H_1,H_2,H_3\in H_2^-(X)$
 be fixed elements with  $H_1H_3=0$. Then, for each pair $(m,k)\in \Z_{>0}\times\Z_{\ge 0}$ with $l=\frac12(m-1-k)\ge 1$
and any $B\in\lay^m_\R$,
the following relation holds:
$$(H_1H_2)(H_3B)\Gamma_{B,k}=\FS_{B,k}+\SS_{B,k},
\text{ where }$$
\scalebox{0.9}{$
\begin{aligned}
&\FS_{B,k}=\frac14\hskip-2mm\sum _{\substack {B_1+B_2=B\\ k_1+k_2=k+1}}\hskip-5mm
(H_1B_1)((H_3B_1)(H_2B_2)-(H_2B_1)(H_3B_2))\binom{l-1}{l_1}\binom{k}{k_1}\Gamma_{B_1,k_1}\Gamma_{B_2,k_2}\\
&\SS_{B,k}=\frac12\sum _{\substack{B_F-\conj_* B_F\\+B_U=B}}\hskip-5mm
(B_UB_F)(H_1B_F)((H_3B_F)(H_2B_U)-(H_2B_F)(H_3B_U))\binom{l-1}{l_U}GW_{B_F}\Gamma_{B_U,k}
\end{aligned}$\qed}
\end{theorem}

In this theorem, notation $\Gamma_{B,k}$ stands for $\Gamma_{m,v,k}$ such that
$B=\frac1d(-mK-v)$ (and similar for $B_i$, $m_i$, $v_i$, etc.).
We let also $l=\frac12(m-1-k)$ (and similar for $m_i$, $k_i$, $l_i$, etc.).
For shortness, we adopt the convention to
put $\Gamma_{\a,k}=0$ for any $\a\notin H_2^-(X)$.

Theorem \ref{main}
allows us to give
 the proof only for one particular
 real del Pezzo surface $X$ for each degree $d=1,2,3$.
 We take $X$ as in Proposition \ref{aux-surfaces} and denote by $G_d$ its automorphism group described there, that is:
\[G_d=\begin{cases}&\Z/2\Z \text{ generated by the Bertini involution }\ \tau, \text{ if } d=1,\\
&\Z/2\Z\times \D4_4 \text{ where } \Z/2 \text{ is generated by the Geiser involution }\ \gamma, \text{ if } d=2,\\
&S_3, \text{ if } d=3.\end{cases}
\]

To apply Theorem \ref{ThirdRelation}, we choose:

\begin{itemize}\item
$H_1\in \lay_\R^1\sm\{-K\}$, so that $H_1^2=KH_1=-1$,
\item
 $H_2=-K$, so that $H_1H_2=1$,
 \item
 $H_3=H_1+H_2$, so that $H_1H_3=0$ as required in Theorem \ref{ThirdRelation},
\end{itemize}
which gives	
\begin{equation}\label{fixed-B-0}
(H_3B) \Gamma_{B,k}=\FS_{B,k}+ \SS_{B,k},\quad
%\Theta^{1}_{B,k}+ \Theta^{2}_{B,k},\quad
\text{where }
\end{equation}
\begin{equation*}\label{SolomonThird}
\begin{aligned}
&\FS_{B,k}
=\frac14\hskip-5mm\sum _{\substack {m_1+m_2=m, m_1,m_2\ge1\\
B_1+B_2=B, B_i\in\lay_\R^{m_i} \\ k_1+k_2=k+1}}\hskip-8mm
(H_1B_1)((H_1B_1)m_2-m_1(H_1B_2))\binom{l-1}{l_1}\binom{k}{k_1}\Gamma_{B_1,k_1}\Gamma_{B_2,k_2}\\
&\SS_{B,k}=
\frac12\hskip-12mm\sum _{\substack{2m_F+m_U=m,\ m_F,m_U\ge1\\
B_F\in\lay^{m_F}, B_U\in\lay_\R^{m_U},\\
B_F-\conj_* B_F+B_U=B}
}
\hskip-12mm(B_UB_F)(H_1B_F)((H_1B_F)m_U-m_F(H_1B_U))\binom{l-1}{l_U}GW_{B_F}\Gamma_{B_U,k}
\end{aligned}
\end{equation*}

Our aim is to perform summation of these identities
over all  $B\in\lay_\R^m$ and all $H_1\in \lay_\R^1\sm\{-K\}$.
As a preliminary step, we
prove a few auxiliary identities which then
will be used repeatedly.
In what follows (similarly to above) we write $H_1$ and $B_i\in\lay^{m_i}$ in a form
$$
H_1=\frac1{d}(-K-w), \quad B_i=\frac1{d}(-m_iK-v_i),\ \text{ where } w, v_i\in K^\perp,
%\ w^2=-d(1+d),
$$
so that $w^2=-d(1+d)$ and both $-K-w$, $-m_iK-v_i$ are divisible by $d$ in $H_2(X)$.
We let also
$$
\w_\R =\{ w\in K^\perp\cap H_2^-(X)
\,|\, w^2=-d(1+d)\, \text{and}\, -K-w\in d\,H^-_2(X)\}.
$$

\begin{proposition}\label{simplifications}\,
For surfaces $X$ as in Proposition \ref{aux-surfaces}, the following holds.

\begin{enumerate}
\item If either $d\in\{1,3\}$, or $d=2$ and $m$ is even, then
$$\Gamma_{m,v,k}=\Gamma_{m,gv,k} \qquad \text{ for any }\quad g\in G_d \quad \text{and any}\quad v\in K^\perp.
$$
If $d=2$, $m$ is odd, and $v\in K^\perp$,
then\hskip5mm  $\Gamma_{m,v,k}=-\Gamma_{m,\gamma v,k}$\ \ \ and
$$\Gamma_{m,v,k}=\Gamma_{m,gv,k}\qquad \text{ for any }\quad g\in \D4_4.
$$
\item $\sum_{w\in\w_\R}w=0$ and $\sum_{H\in \lay_\R^1\sm\{-K\}} H=-\frac1{d}\card\{\w_\R\}K.$
\item
If either $d\in\{1,3\}$ or $m$ is even, then for any $H\in \lay_\R^1\sm\{-K\}$ we have
$$
\sum_{B\in\lay_\R^{m}} \Gamma_{B,k}B=
-\frac{m}{d}\Gamma_{m,k}K, \,\, \sum_{B\in\lay_\R^{m}} (HB)\Gamma_{B,k}=\frac{m}{d}\Gamma_{m, k}.
$$
If $d=2$ and $m$ is odd, then the above sums vanish.
\item\footnote{Properties (4) and (5) hold without any assumption on a real structure of $X$.}
For any $v\in K^\perp$, $v\ne0$, we have
$$
\sum_{B\in\lay^{m}}B(vB)GW_{B}=\frac1{v^2}\sum_{B\in\lay^{m}}v(vB)^2GW_{B}.
$$
\item
For every
$u,v\in K^\perp$, we have
$$
u^2\sum_{B\in \lay^{m}} (vB)^2 GW_{B}= v^2\sum_{B\in \lay^{m}} (uB)^2 GW_{B}.
$$
\end{enumerate}
\end{proposition}

\begin{proof}
Proof of the first claim is a direct combination of preservation
of the arithmetic genus $g_a$ and Welschinger numbers by the action of $g\in G_d$ with the
properties of their action on $\hat q$ described in Proposition \ref{aux-surfaces}.

To prove Claims (2) and (3), except the vanishing statement in (3), it is sufficient to notice that the expressions $\sum w$, $\sum H$, and $\sum \Gamma_{B,k}B$ are invariant under the action of $G_d$, hence proportional to $K$ (see Proposition \ref{aux-surfaces}),
and that to determine the coefficient of proportionality it remains to take scalar product with $K$. For proving the vanishing statement, we notice that, as it follows from Claim (1), if for an arbitrary chosen $B\in \lay^m_\R$ the number $\Gamma_{m,k}$ is not zero, then the orbit of $B$ under the action of $\Z/2\times \D4_4$ splits into 2 orbits
of $\D4_4$ that are interchanged by $\gamma$. This implies (using the same arguments as above) that
\begin{equation*}
\begin{aligned}
\sum_{g\in\Z/2\times \D4_4}\Gamma_{gB,k}B=\sum_{g\in \D4_4}\Gamma_{gB,k} gB+ \sum_{g\in \D4_4}\Gamma_{\gamma g B,k}\gamma g B=
\\
\Gamma_{B,k}(\sum_{g\in \D4_4} gB-\gamma(\sum_{g\in \D4_4} gB))=\lambda K -\gamma\lambda K=0.
\end{aligned}
\end{equation*}

To prove Claims (4) and (5) we notice that due to the invariance of Gromov-Witten numbers $GW_B$ under the Weyl group action on $K^\perp=E_{9-d}$ (which follows from
 to the monodromy invariance of Gromov-Witten numbers, and interpretation of the Weyl group as monodromy), these claims become straightforward consequences of the fact that the function $K^\perp \to \R$ given by $\alpha\mapsto \alpha^2$ is the only, up to scalar factor, quadratic function invariant under the action of the Weyl group
(as it follows from the irreducibility of the action, see \cite[§2.1, Prop. 1]{Bour}). \end{proof}

Next, we note that in
the left-hand side of (\ref{fixed-B-0}) we have
$(H_3B)\Gamma_{B,k}=((\frac1d+1)m-\frac1d wB)\Gamma_{B,k}$. Then, denoting
by $\Theta_{m,k}^{(i)}$ the sum of $\Theta_{B,k}^{(i)}$ over all $H_1\in \lay_\R^1\sm\{-K\}$ and $B\in\lay_\R^m$,
we conclude that
\begin{equation}\label{sum-B-0}
(\frac1d+1)\card(\lay_\R^1\sm \{-K\})
m\Gamma_{m,k}=\FS_{m,k}+\SS_{m,k}
\end{equation}
since the terms $(wB)\Gamma_{B,k}$ vanish after summation
due to Proposition \ref{simplifications}(2).

\subsection{Proof of Theorem \ref{two-recursive-formulas} for odd $\mathbf{m=2n+1\ge3}$}\label{recursion}
Due to Proposition \ref{vanishing-in-degree-2}(2), this part of Theorem \ref{two-recursive-formulas}
holds trivially for $d=2$. Therefore, here we may assume (for simplicity) that $d\ne 2$.

We let $k=0$ and have

\begin{equation*}\label{1st-summand}
\begin{aligned}
\FS_{m,0}
&= \frac14 \sum_{\substack{m_1+m_2=m,m_i\ge 1 \\B_i\in\lay_\R^{m_i}, H_1\in\lay_\R^1}}
\binom{l-1}{l_1}
(H_1B_1)((H_1B_1)m_2-(H_1B_2)m_1)\Gamma_{B_1,0}\Gamma_{B_2,1}
\\
&=\frac14\sum_{H_1\in \lay_\R^1}\ \sum_{\substack{m_1+m_2=m\\ m_1,m_2\ge 1}}
\sum_{\ B_2\in\lay_\R^{m_2}}\binom{l-1}{l_1}m_2\Gamma_{B_2,1}\sum_{B_1\in\lay_\R^{m_1}}(H_1B_1)^2\Gamma_{B_1,0}
\\
&-\frac14\sum_{H_1\in\lay_\R^1}\sum_{\substack{m_1+m_2=m\\ m_1,m_2\ge 1}}\binom{l-1}{l_1}m_1\big(\sum_{B_1\in\lay_\R^{m_1}}(H_1B_1)\Gamma_{B_1,0}\sum_{B_2\in\lay_\R^{m_2}} (H_1B_2)\Gamma_{B_2,1}\big),
\end{aligned}
\end{equation*}
\begin{equation*}
\small
\begin{aligned}
\SS_{m,0}
&=\frac12\hskip-8mm\sum_{\substack {2m_F+m_U=m\\ m_F, m_U\ge 1 \\B_F\in\lay^{m_F}, B_U\in\lay_\R^{m_U},
H_1\in\lay_\R^1}}\hskip-8mm
(B_UB_F)(H_1B_F)((H_1B_F)m_U-(H_1B_U)m_F)\binom{l-1}{ l_U} GW_{B_F}\Gamma_{B_U,0}
\\
&=\frac12\hskip-1mm\sum_{\substack{2m_F+m_U=m\\ m_F, m_U\ge 1 \\ H_1\in\lay_\R^1}}\hskip-1mm
\big(\sum_{B_U\in\lay^{m_U}}
m_UB_U\binom{l-1}{ l_U}
\Gamma_{B_U,0}\  \centerdot  \sum_{B_F\in\lay^{m_F}}
B_F(H_1B_F)^2 GW_{B_F}\big)
\\
&-\frac12\hskip-1mm\sum_{\substack{2m_F+m_U=m\\ m_F, m_U\ge 1\\H_1\in\lay_\R^1}}\hskip-1mm
\big( \sum_{B_U\in\lay_\R^{m_U}}B_U(H_1B_U)\binom{l-1}{ l_U} \Gamma_{B_U,0}\ \centerdot \sum_{B_F\in\lay^{m_F}}m_FB_F(H_1B_F) GW_{B_F}\big).
\end{aligned}
\end{equation*}
In the last relation, symbol ``$\centerdot$''  is used to denote the intersection index in $H_2(X)$.

Substituting $H_1=\frac1{d}(-K-w)$ and observing the vanishing (due to Proposition \ref{simplifications}(2))
of summands where the factors $(wB_i)$ enter linearly,
we obtain
\[
\FS_{m,0} =
\frac1{4d^2}\sum_{w\in\w}\big(\sum_{\substack{m_1+m_2=m\\ m_i\ge 1,B_i\in\lay_\R^{m_i}} }\hskip-3mm\binom{l-1}{l_1}
\big(m_2(wB_1)^2-m_1(wB_1)(wB_2)\big)\Gamma_{B_1,0}\Gamma_{B_2,1}\big).
\]
Next,
we note that the terms with a factor $(wB_2)\Gamma_{B_2,1}$ vanish after summation
over $B_2$, since by Proposition \ref{simplifications}(3), $\sum B_2\Gamma_{B_2,k}$ is collinear with $K$, which is orthogonal to $\omega$. Thus, we get
\begin{equation}\label{fin-1}
\FS_{m,0}
=\frac1{4d^2}\sum_{\substack{m_1+m_2=m\\ m_1, m_2\ge 1}}\hskip-2mm\big(\binom{l-1}{l_1}m_2\Gamma_{m_2,1}\sum_{\substack{B_1\in\lay_\R^{m_1}\\
w\in\w_\R}}(wB_1)^2\Gamma_{B_1,0}\big).
\end{equation}

Similarly, after the same substitution for $H_1$ into the first summand of $\SS_{m,0}$,
we apply
Proposition \ref{simplifications}(3)
to the factor containing $\sum{B_U}\Gamma_{B_U,0}$
and obtain
\begin{equation*}
\frac{1}{2d^3}  \sum_{\substack{w\in\w_\R\\ 2m_F+m_U=m\\ m_F, m_U\ge 1}}\binom{l-1}{ l_U} m_U^2m_F
\Gamma_{m_U,0} \sum_{\substack{B_F\in\lay^{m_F}}} (m_F-(wB_F))^2GW_{B_F}.
\end{equation*}
Then, cancelation in $(m_F-(wB_F))^2$ of linear in $w$ term (since $\sum_{w\in\w_\R}w=0$), gives
\begin{equation*}
\frac1{2d^3}\sum_{\substack{w\in\w_\R\\ 2m_F+m_U=m\\ m_F, m_U\ge 1}}\binom{l-1}{ l_U} m_U^2\Gamma_{m_U,0} \big(m_F^3N^{GW}_{m_F}+
m_F\sum_{\substack{B_F\in\lay^{m_F}
}}(wB_F)^2 GW_{B_F}\big).
\end{equation*}
% For
In the second summand of   $\SS_{m,0}$, after the same substitution for $H_1$ and cancelation of linear in $w$ terms,
we apply Proposition \ref{simplifications}(4) (with the choice $v=w$) and obtain
$$\begin{aligned}
&-\frac1{2d^3}\hskip-3mm\sum_{\substack{w\in\w_\R\\ 2m_F+m_U=m\\ m_F, m_U\ge 1}}
\hskip-5mm \binom{l-1}{ l_U} m^3_Fm^2_U \Gamma_{m_U,0}N^{GW}_{m_F}
\\
&+ \frac1{2d^3(1+d)}\hskip-3mm\sum_{\substack{w\in \w_\R\\ 2m_F+m_U=m\\ m_F, m_U\ge 1}}\hskip-3mm
\big(
\sum_{B_U\in\lay_\R^{m_U}} \hskip-3mm
\binom{l-1}{ l_U}(wB_U)^2\Gamma_{B_U,0}\sum_{B_F\in\lay^{m_F}} m_F(wB_F)^2 GW_{B_F}
\big).
\end{aligned}
$$
Afterwards we cancel two opposite terms in the sum and conclude that
\begin{equation}
\begin{aligned}\label{fin-2}
\small
&\SS_{m,0}
=\frac1{2d^3}\hskip-5mm\sum_{\substack{w\in\w_\R\\ 2m_F+m_U=m\\ m_F, m_U\ge 1}}\hskip-5mm(\binom{l-1}{ l_U} m_Fm_U^2\Gamma_{m_U,0}\sum_{\substack{B_F\in\lay^{m_F}
}}(wB_F)^2 GW_{B_F}\big)+\\
& \frac1{2d^3(1+d)}\hskip-5mm\sum_{\substack{w\in \w_\R\\ 2m_F+m_U=m\\ m_F, m_U\ge 1}}\hskip-5mm
\big(\hskip-1mm\sum_{B_U\in\lay_\R^{m_U}} \hskip-3mm
\binom{l-1}{ l_U}(wB_U)^2\Gamma_{B_U,0}
\sum_{B_F\in\lay^{m_F}}\hskip-2mm m_F(wB_F)^2 GW_{B_F}\big).
\end{aligned}
\end{equation}

Now, we substitute the expressions obtained in (\ref{fin-1}) and (\ref{fin-2}) into (\ref{sum-B-0})
and observe that, as it follows from Proposition \ref{miracle} where we transform $e$ into $w$ in accordance with Proposition \ref{simplifications}(5), the second term in (\ref{fin-2}) cancels  (\ref{fin-1}).
In this way we get
\begin{equation*}
\begin{aligned}
&\frac{1+d}{d}
\card(\lay_\R^1\sm \{-K\})
m\Gamma_{m,0}=
\\
&\frac{d(1+d)}{4d^3}\card \w_\R \, \sum_{\substack{2m_F+m_U=m\\ m_F, m_U\ge 1}}\big( m_Fm_U^2\binom{l-1}{ l_U} \Gamma_{m_U,0}\sum_{\substack{B_F\in\lay^{m_F}
}}(eB_F)^2 GW_{B_F}\big).
\end{aligned}
\end{equation*}
Finally, division by $\frac{1+d}{2d^3}\card(\lay_\R^1\sm \{-K\})=\frac{1+d}{2d^3}\card \w_\R $ and using once more Proposition \ref{miracle} gives
\begin{equation*}
2dm\Gamma_{m,0}=
\frac12\sum_{\substack{2m_F+m_U=m\\ m_F, m_U\ge 1}}\big( m_Fm_U^2\binom{l-1}{ l_U} \Gamma_{m_U,0}
(-2\Gamma_{2m_F,1})\big)
\end{equation*}
which proves the first relation of Theorem \ref{two-recursive-formulas} (cf. its reformulation (\ref{recurrence-via-Gamma})). \qed

\subsection{Proof of Theorem \ref{two-recursive-formulas} for $\mathbf{m=2n+2\ge4}$}. We
let $k=1$ and have:

\begin{equation}\label{gen-fixed-B}
\small
\frac{1+d}{d}\card(\lay_\R^1\sm \{-K\})m  \Gamma_{m,1}=
\FS_{m,1}+ \SS_{m,1},
\end{equation}

\begin{equation}\label{gen-one}
\begin{aligned}
\FS_{m,1}&= \frac14\hskip-5mm \sum_{\substack {m_1+m_2=m,m_i\ge 1 \\B_i\in\lay_\R^{m_i}, H_1\in\lay_\R^1}}
\hskip-1mm\binom{l-1}{l_1}
(H_1B_1)((H_1B_1)m_2-(H_1B_2)m_1)\Gamma_{B_1,1}\Gamma_{B_2,1}
\\
&=\frac1{4d^4}\hskip-3mm\sum_{m_1+m_2=m, m_i\ge 1}\big(\binom{l-1}{l_1} m_2\Gamma_{m_2,1}\sum_{w\in \w_\R,
 B_1\in\lay_\R^{m_1}}(wB_1)^2\Gamma_{B_1,1}\big),
\end{aligned}
\end{equation}

\begin{equation}\label{gen-two}
\small
\begin{aligned}
\SS_{m,1}
&=\frac12\hskip-9mm\sum_{\substack {2m_F+m_U=m\\ m_F, m_U\ge 1\\B_F\in\lay^{m_1}, B_U\in\lay_\R^{m_2}\\ H_1\in\lay_\R^1}}
\hskip-9mm(B_UB_F)(H_1B_F)((H_1B_F)m_U-(H_1B_U)m_F)\binom{l-1}{ l_U} GW_{B_F}\Gamma_{B_U,1}
\\
&=\frac{d}{2d^6}\hskip-1mm\sum_{\substack{2m_F+m_U=m\\ m_F, m_U\ge 1}}
\hskip-1mm\binom{l-1}{ l_U} m_Fm_U^2\Gamma_{m_U,1}\sum_{w\in \w_\R,
v_F\in K^\perp}
(w v_F)^2 GW_{B_F} -
\\
&-\frac1{2d^6}\sum_{\substack{2m_F+m_U=m\\ m_F, m_U\ge 1\\w\in \w_\R,
 v_F\in K^\perp, v_U\in K^\perp\cap H_2^-}}
\binom{l-1}{ l_U}(v_Fv_U)(wv_F)(wv_U) GW_{B_F}\Gamma_{B_U,1}.
\end{aligned}
\end{equation}
We substitute (\ref{gen-one}) and (\ref{gen-two}) into (\ref{gen-fixed-B}),
apply
Proposition \ref{simplifications}(5) where we choose  $u=w$ and $v=e$ with $e^2=-2, e\in K^\perp$ to perform a transformation
$$
\begin{aligned}
& \sum_{v_F\in K^\perp} (wv_F)^2 GW_{B_F}= d^2 \sum_{B_F\in\lay^{m_F}} (wB_F)^2 GW_{B_F}=\\
&= \frac{d^3(1+d)}2 \sum_{B_F\in\lay^{m_F}} (eB_F)^2 GW_{B_F}\overset{\rm Prop.\ref{miracle}}{=}
- d^3 (1+d) \Gamma_{2m_F,1},
\end{aligned}
$$
cancel similar terms in $\FS_{m,1}+ \SS_{m,1}$ and get
\begin{equation}
\frac{1+d}{d}m\Gamma_{m,1}= -\frac{d^3(1+d)}{2\cdot d^5} \sum_{\substack{2m_F+m_U=m\\ m_F, m_U\ge 1}}
\binom{l-1}{ l_U} m_Fm_U^2\Gamma_{m_U,1}\Gamma_{2m_F,1}
\end{equation}
wherefrom the required recursion relation (cf. (\ref{recurrence-via-Gamma}))
$$
2md\Gamma_{m,1}+ \sum_{j=1}^{n}\big(\binom{n-1}{n-j}j(m-2j)^2\Gamma_{m-2j,1}\Gamma_{2j,1}=0\quad\text{for $m=2n+2$}. \qed
$$

\subsection{Explicit formulas}\label{explicit}
The recursive formulas of Theorem \ref{two-recursive-formulas} have a surprisingly simple explicit solution.
\begin{theorem}\label{explicit-theorem} For each $d=1,2,3$ and any $n\in\Z_{\ge 0}$, we have
\[
N_{2n+1,0}\, =\,\frac14 N_{1,0}\, b^n\, (n+\frac12)^{n-2},\quad  N_{2n+2,1}\, =\, N_{2,1}\,b^{n}\, (n+1)^{n-2}\quad \text{ with}\quad b=\frac{4N_{2,1}}{d}.
\]
\end{theorem}

\begin{proof}
Both relations stated in Theorem \ref{two-recursive-formulas} hold trivially
for these
values if
$n=0$.
To check the second of these two relations for $n>0$
we substitute there the given values of $N_{2k,1}$, $k\le n+1$ and get
$$
\begin{aligned}
&(2n+2)d\,N_{2,1}\,b^{n}\,(n+1)^{n-2}
= \\
&2\sum_{j=1}^{n}\binom{n-1}{n-j}j(2n+2-2j)^2
N_{2,1}\,b^{n-j}\, (n-j+1)^{n-j-2}
N_{2,1}\,b^{j-1}\, j^{j-3}.
\end{aligned}
$$
division by $2d\,N_{2,1}b^{n}=8N_{2,1}^2b^{n-1}$
it gives
$$
(n+1)^{n-1}
= \sum_{j=1}^{n}\binom{n-1}{n-j}(n-j+1)^{n-j} j^{j-2}
$$
which is a special case of  Abel's binomial theorem (see, for example, \cite{Comtet})

$${\frac {(x+y)^{m}}{x}}=\sum _{k=0}^{m}{\binom {m}{k}}(x-kz)^{k-1}(y+kz)^{m-k}$$
where we put $x=-z=1$, $y=n$, $m=n-1$, $k=j-1$.

Substituting the given values into the first relation we obtain
$$
\begin{aligned}
&(2n+1)d\,\frac14 N_{1,0}\, b^n\, (n+\frac12)^{n-2}=\\
&2\sum_{j=1}^{n}\binom{n-1}{n-j}j(2n+1-2j)^2
\frac14 N_{1,0}\, b^{n-j}\, (n-j+\frac12)^{n-j-2}
N_{2,1}\,b^{j-1}\, j^{j-3}.
\end{aligned}
$$
Division by
$\frac12d\, N_{1,0}b^n=2N_{1,0}\,b^{n-1}N_{2,1}$
turns it into
$$
(n+\frac12)^{n-1}=
\sum_{j=1}^{n}\binom{n-1}{n-j}(n-j+\frac12)^{n-j} j^{j-2}
$$
which is Abel's binomial relation
for $x=-z=1$, $y=n-\frac12$, $m=n-1$, $k=j-1$.
\end{proof}

\begin{cor}
For real del Pezzo surfaces $X$ of degree $d=1$ or $d=2$
whose real locus $X_\R$ has the maximal number of connected components, the invariant
$N_{dm,k}$ coincides with the
Welschinger invariant $W_{-mK,k}$.
In particular,
for $d=1$ and $X_\R=\Rp2\dsum 4S^2$ we have
 $$\begin{aligned}
W_{-(2n+1)K,0}=& 2\,(120)^n(n+\frac12)^{n-2},\\
W_{-(2n+2)K,1}=&30\,(120)^n(n+1)^{n-2},
\end{aligned}$$
and for $d=2$ and $X_\R=\dsum4S^2$
$$
W_{-(n+1)K,1}=6\,(12)^n(n+1)^{n-2}.
$$
\end{cor}

\begin{proof} This is immediate from Theorem \ref{explicit-theorem} and knowledge of values of $N_{1,0}, N_{2,1}$ (see Table below), since under the assumptions imposed on $X_\R$ the eighenspace $\ker (1+\conj_*)\subset H_2(X)$ is generated by $K$
and, in addition, $\hat q(mK)-m^2=m^2\hat q( K)-m^2=0$.
\end{proof}

\subsection{Few first values and simple applications}\label{applications}
To begin, we list the values of $N_{m,k}$ for $m\le 6$ obtained by means of the formulas established above.
The initial values $N_{1,0}, N_{2,1}$ we use there:
\begin{itemize}\item
for $K^2=1$, were obtained in \cite{TwoKinds}, \cite{Combined} via some ``lattice calculation'',
\item
for $K^2=2$, can be obtained in a similar way,
\item
for $K^2=3$,
the value
$N_{1,0}=3$ was explained in \cite{F-K}, and
$N_{2,1}=N_{1,0}$ follows from a natural one-to-one correspondence $\lay^1_\R\to \lay^2_\R$ assigning to a line the divisor class of
residual conics of hyperplane sections containing this line.
\end{itemize}

\resizebox{1\textwidth}{!}{\hbox{
\vtop{\hsize70mm\hskip25mm $N_{m,k}$\newline\boxed{
\begin{tabular}{c|c|c|c}
(m,k)&$K^2=1$&$K^2=2$&$K^2=3$\\
\hline
(1,0)&8&0&3\\
(2,1)&30&6&3\\
(3,0)&160&0&2\\
(4,1)&1800&36&6\\
(5,0)&28800&0&12\\
(6,1)&432000 &864&48\\
\end{tabular}}}
\hskip8mm
\vtop{\hsize70mm\hskip25mm$|\Gamma_{m,k}|$\newline
\boxed{
\begin{tabular}{c|c|c|c}
(m,k)&$K^2=1$&$K^2=2$&$K^2=3$\\
\hline
(1,0)&16&0&6\\
(2,1)&60&12&6\\
(3,0)&160&0&2\\
(4,1)&1800&36&6\\
(5,0)&14400&0&6\\
(6,1)&216000 &432&24\\
\end{tabular}}}
}}

\begin{proposition}\label{cubic-3} For any nonsingular real cubic surface $X$:
\begin{itemize}
\item The signed count of 2-points-constrained real twisted cubics is equal to $3-\chi(X_\R)$, independently on the number of real points in the points-constraint.
When the surface is maximal, it contains $h=40$ hyperbolic {\rm (}with $\hat q=1${\rm )} and $e=32$ elliptic {\rm (}with $\hat q=-1${\rm )} twisted cubics.
\item The signed count of 3-points-constrained real non-singular rational quartic curves is equal to $9-3\chi(X_\R)$, independently on the number of real points in the points-constraint. When the surface is maximal, it contains $h=120$ hyperbolic {\rm (}with $\hat q=0${\rm )} and $e=96$ elliptic {\rm (}with $\hat q=2${\rm )} quartics.
\end{itemize}
\end{proposition}

\begin{proof} The third layer is formed by $-K$ and divisor classes of twisted cubics. Therefore, the first statement follows from $N_{3,k}=2-k$ and
$W_{-K,k}=\chi(X_\R)-(k+1)$.
For a separate count of hyperbolic and elliptic twisted cubics,
it is sufficient to notice, in addition, that their number is equal to the number of roots, $72$, in $E_6$.

The fourth layer is formed by divisor classes representable by non-singular rational quartic curves (that is by divisor classes of form $-2K-L'-L''$ where $L', L''$ any pair of disjoint lines) and genus 1 quartic curves (that is by divisor classes $-K+L$ where $L$ is any line). Therefore, the second statement follows from $N_{4,k}=9-4K$,
$W_{-K+L, k}= \chi(X_\R)-k$ and $\hat q(-K+L)=\hat q(L)-1$. For a separate count of hyperbolic and elliptic quartics, we note that the number of pairs of disjoint real lines
on a maximal cubic surface is equal to $\frac{27\cdot 16}2=216$.
\end{proof}

\begin{remark} In the case of cubic surfaces, there are natural bijections: between the set of divisor classes of twisted cubics (resp. non-singular rational quartic curves)
and the set of isomorphism classes of presentations of the surface as 6-blowup of $\P^2$  (resp. as 5-blowup of $\P^1\times\P^1$). Therefore, cited counts can be interpreted as signed counts of the corresponding blow-up models.
\end{remark}

Considering real del Pezzo surfaces $X$ with $K^2=2$ as double coverings of $\P^2$ branched along real non-singular quartic curves $A\subset \P^2$
and using skew-invariance of $\hat q$ under the deck transformation,
we translate our signed count of curves belonging to even layers $\lay^{2n}_\R$ into a signed count of real rational point-constrained curves of degree $n$
tangent to $A$ at $2n-1$ points plus a signed count of real rational point-constrained curves of degree $2n$ tangent to $A$ at $4n$ points. In particular, such
a point-constrained curve of degree $2n$ is called {\it hyperbolic} (resp. {\it elliptic}), if $\hat q$ takes value $0$ (resp. $2$) on its lifts to $X$.

\begin{proposition}\label{quartics}
Let $A\subset \P^2$ be a real non-singular quartic curve and $\varOmega$ one of two halves of $\Rp2$ bounded by $A_\R$. Then:
\begin{itemize}
\item The signed count of real conics 4-tangent to $A$ and constrained by a point in $\varOmega$ gives $8-2\chi(\varOmega)$.
If $A$ is maximal and $\varOmega$ is non-orientable,
then this count involves $h=70$
hyperbolic  and $e=56$
elliptic  conics.
\item If $A$ is maximal and  $\varOmega$ is non-orientable, then signed count of real non-singular rational quartic curves 8-tangent to
with $k$ real point-constraints chosen in $\varOmega$ gives
336 if $k=1$ and $896$ if $k=3$.
\end{itemize}
\end{proposition}

\begin{proof} The second layer, $\lay^2$, is formed by $-K$ and the divisor classes of lifts of 4-tangent conics.
Therefore, the first statement follows from $N_{2,1}=6$ and $W_{-K,1}=\chi(X_\R)-2=2\chi(\varOmega)-2$.
In the maximal case, $8-2\chi(\varOmega)=14$ and the number of 4-tangent conics is the number of roots, 126,
in $E_7$.

The fourth layer is formed by $-2K$ and the divisor classes of lifts of 8-tangent quartics. Here we obtain the required number of quartics,
$4\cdot224$  if $k=3$ and $2(132+36)$ if $k=1$, due to
$N_{4,k}=0$, $W_{-2K,k}=-224$ for  $k=3$ and
$N_{4,k}=36$,  $W_{-2K,k}=-132$ for  $k=1$ (see \cite{IKS} for these values of $W_{-2K,k}$).
\end{proof}

Considering real del Pezzo surfaces $X$ with $K^2=1$ as double coverings $\pi : X\to Q$ of a real quadric cone $Q\subset \P^2$ branched along real non-singular sextics $C\subset Q$
and using invariance of $\hat q$ under the deck transformation,
we translate our signed count of curves belonging to the second layer $\lay^{2}_\R$ into a signed count of real rational 1-point-constrained hyperplane sections
tangent to $C$ at $2$ points plus a signed count of real 1-point-constrained quartics tangent to $C$ at $6$ points (and obtained as transversal sections of $Q$ by quadrics). In particular, such
a point-constrained quartic is called {\it hyperbolic} (resp. {\it elliptic}), if $\hat q$ takes value $0$ (resp. $2$) on its lifts to $X$.

\begin{proposition}\label{conics}
For a real non-singular sextic $C\subset Q$,
the signed count of real quartics 6-tangent to $C$ constrained by a point in $\pi (X_\R)$ gives
$6+\frac{\chi^2(X_\R)-1}2$.
If $X$ is maximal and $X_\R$ is connected, then
this count involves $h=1192$ hyperbolic and $e=1208$ elliptic quartics.
\end{proposition}

\begin{proof} Follows from $N_{2,1}=30$ and $W_{-2K,1}=6+\frac{\chi^2(X_\R)-1}2$
(see \cite{Combined}).
\end{proof}

\section{Concluding remarks}\label{concluding}

\subsection{Generating functions}
Recall the {\it tree function} $T(x) =\sum_{n\ge1}n^{n-1}\frac{x^n}{n!}$.
\begin{proposition} For each $d=1,2$ and $3$, the functions
$$
N^{\text{even}}(x)=\sum_{n\ge0}N_{2n+2, 1}\frac{x^n}{n!}, \quad N^{\text{odd}}(x)=\sum_{n\ge0}N_{2n+1, 0}\frac{x^n}{n!}
$$
can be expressed through $T(x)$ as
$$
N^{\text{even}}(x)=N_{2,1}(bx)^{-1}(T(bx)-\frac12 T^2(bx)), \quad N^{\text{odd}}(x)= N_{1,0}(bx)^{-\frac12}( T^\frac12 (bx) - \frac13 T^\frac32 (bx)),
$$
where
$b$ is 120, 12 and 4 for $d=1,2,3$ respectively.
\end{proposition}

\begin{proof} After substitution of the values of $N_{2n+2,1}$ from Theorem \ref{explicit-theorem}
we obtain
$$\begin{aligned}
N^{\text{even}}(x)&
=
N_{2,1}\sum_{n\ge0} (n+1)^{n-1}\frac{(bx)^n}{(n+1)!}=
N_{2,1} G(bx),\quad\text{where}\\
G(x)&=
\sum_{n\ge1}n^{n-2}\frac{x^{n-1}}{n!}=\frac1x(T(x)-\frac12T^2(x)).
\end{aligned}$$
For the latter identity, see, {\it i.e.}, \cite{Moon}.
For values of $b=\frac{4N_{2,1}}d$  see Section \ref{applications}.

A similar substitution for $N_{2n+1,1} $ gives
$$\begin{aligned}
N^{\text{odd}}(x)&=
\sum_{n\ge 0} \frac14 N_{1,0}\, b^n\, (n+\frac12)^{n-2}\frac{x^n}{n!}=
N_{1,0}Q(bx)
\quad\text{with}\\
Q(x)=\sum_{n\ge 0}\frac14 (n+\frac12)^{n-2}\frac{x^n}{n!}&= x^{-\frac12}\int_0^x\frac12 x^{-\frac12}e^{\frac12{T(x)}} dx=
x^{-\frac12}( T^\frac12 (x) - \frac13 T^\frac32 (x))
\end{aligned}
$$
where the last equalities follow from the following relations (see, f.e.,  \cite{C-K})
$$
\sum\frac12 (n+\frac12)^{n-1}\frac{x^n}{n!}= e^{\frac12{T(x)}}=\big(\frac{T(x)}x\big)^\frac12.
$$
\end{proof}

\subsection{On Gromov-Witten side}\label{onside}
It is thought-provoking that over the complex field, for any del Pezzo surface, the sums of genus-0 Gromov-Witten invariants over layers $\lay^m$  satisfy essentially the same recursion relation as genus-0 Gromov-Witten invariants
of projective plane.
\begin{proposition}\label{KMlikePlus}
{\it For every del Pezzo surface of degree $1\le d\le 6$ and every $m\ge 4$,}
\begin{equation}\label{KMlayer}
d^2 N_{m}^{GW}= \sum_{\substack {m_1+m_2=m \\ m_1,m_2\ge 1}} N^{GW}_{m_1}N^{GW}_{m_2}m_1^2m_2\big( m_2\binom{m-4}{m_1-2}-m_1\binom{m-4}{m_1-1}\big)
\end{equation}
where $N_{m}^{GW}$ stands for $\sum_{\alpha\in\lay^m(X)} GW(\alpha)$.
The initial values of
$N_{m}^{GW}$ for  $m=1,2,3$ and $1\le d\le 6$ are shown in the table below.
\footnote{$N^{GW}_1$ and $N^{GW}_2$ are obtained by a direct computation, while the values of $N^{GW}_3$
are computed by means of another recursion relation,  $6N^{GW}_3= \frac{4}{K^2} N^{GW}_2 S^{GW}_1 + \frac12 S^{GW}_1S^{GW}_2$, where $S^{GW}_n$ stands for
$\sum_{\alpha\in\lay^n} (e\alpha)^2 GW_{\alpha}$ computed in Prop. \ref{frmlar-moments}.
%which is also a particular incarnation of the WDVV-associativity relation.
}

\boxed{
\begin{tabular}{c|c|c|c|c|c|c}
&d=1&d=2&d=3&d=4&d=5&d=6\\
\hline
$N^{GW}_1$&$252$&$56$&$27$&$16$&$10$&$6$\\
$N^{GW}_2$&$5130$&$138$&$27$&$10$&$5$&$3$\\
$N^{GW}_3$&$446400$&$1248$&$84$&$16$&$5$&$2$\\
\end{tabular}}

\end{proposition}

It is not difficult to show that this implies the following {\it square root rule}.
\begin{cor}\label{growth} For any $1\le d\le 6$ one has
$\log N_{2m}^{GW}= 2m\log m+O(m)$,  while $\log N_{2m,1}$ for $d=1,2,3$ and
$\log N_{2m+1,0}$ for $d=1,3$ are of type $m\log m+O(m)$ $($recall that $N_{2m+1,0}=0$ for $d=2${\rm ).}
\qed\end{cor}

In the proof of Proposition \ref{KMlikePlus} we use (in a similar, but much simpler, way as in out proof of recursion relations for $N_{m,k}$) the fact that in the case of $d\le 6$ the action of the Weyl group on $K^\perp$
has no any nonzero invariant element. Note that, as an elementary check shows, the relation (\ref{KMlayer}) does not hold for $d=7$ and $8$. When $d=9$, it turns
(after replacing $m_i$ by $3d_i$) into the celebrated Kontsevich-Manin recursion relation.

The next formula is an immediate consequence of Proposition \ref{miracle} combined with Theorem \ref{explicit-theorem}.

\begin{proposition}\label{frmlar-moments}
$
\sum_{\alpha\in\lay^n} (e\alpha)^2 GW_{\alpha}=
2da^nn^{n-3}
$
where $e$ is an arbitrary class $e\in K^\perp$ with $e^2=-2$ 
and
$a$ equals $60$ if $d=1$, $6$ if $d=2$, and $2$ if $d=3$.
\qed\end{proposition}

\subsection{Other del Pezzo surfaces}\label{other}
There are two more kinds of del Pezzo surfaces to which our approach applies almost literally.
One of them is  real nonsingular quadrics $X\subset P^3$ with $X_\R\ne\varnothing$.
Such non-empty quadrics form two real deformation classes:
one with $X_\R=S^2$ and another with $X_\R=S^1\times S^1$.
Pick a $\Spin$-structure on $X_\R$ induced from a $\Spin$-structure on $\Rp3$,
and let  $q:H_1(X_\R;\Z/2)\to\Z/2$ denote the associated
quadratic function.
We put in this case
 $N_{m,k}=\sum_{A\in\Cal C_\R(m,k,\mathbf x)}
 (-1)^{q(A_\R)}w(A)$ and obtain the following results.

\begin{itemize}\item
The numbers $N_{m,k}$ are independent of the choice of
a real quadric
$X$ with $X_\R\ne\varnothing$
and a $\Spin$-structure on $\Rp3$.
\item
The numbers $N_{m,k}$ vanish unless $m=4n$, while $N_{4n,k}$ is equal
to the Welschinger invariant $W_{nh,k}$, $h=-\frac12 K$, of
$X$ with $X_\R=S^2$ (see \cite{W-optim}, \cite{IKS-small}, \cite{Br-P}, \cite{Chen-Z} for various methods of calculation of $W_{nh,k}$).
\item
The relation $N_{4n,1}=
2^{2n-3}\sum_{\a\in\lay^{2n}}(\a e)^2GW_{\a}$ (where $e$ is an arbitrary class $e\in K^\perp$ with $e^2=-2$) holds for any $n\ge1$.
\end{itemize}

Another example is provided by del Pezzo surfaces $X$ of degree 4 presented as double coverings of real non-singular quadrics $Q\subset \P^3$ branched along real non-singular curves representing
the doubled hyperplane section class.
The covering $X\to Q$
is naturally embedded into the quadratic cone $Z\subset \P^4$ over $Q$. Thus, we can pick a $\Pin^-$-structure on $X_\R$ induced from a $\Pin^-$-structure on the non-singular part of $Z_\R$ and
put $N_{m,k}= \sum_{A\in \Cal C_\R(m,k,\bold x)}
i^{\hat q([A])-m^2}w(A)$, as usual.

\begin{itemize}\item
The numbers $N_{m,k}$ are preserved under
change of the branching curve.
\item
The numbers $N_{m,k}$ vanish if either $m$ is odd or $k>1$.
\item
The relation $N_{2n,1}=
2^{n-3}\sum_{\a\in\lay^{n}}(\a e)^2GW_{\a}$ (where $e$ is an arbitrary class $e\in K^\perp$ with $e^2=-2$ ) holds for any $n\ge1$.
\end{itemize}

\end{document}